\documentclass[10pt,reqno]{amsart}

\usepackage{amssymb,amsthm,amsmath,mathtools}
\usepackage{mathdots,bm}
\usepackage[left=2.5cm, right=2.5cm, top=3cm, bottom=3cm]{geometry}
\usepackage{titlesec}
\titleformat{\subsection}[hang]{\normalfont\bfseries}{\thesubsection}{1em}{}
\usepackage{secdot}
\usepackage{ifthen}
\usepackage{mathabx}
\usepackage{hyperref}
\usepackage{float}
\usepackage{import}
\usepackage{graphicx}   
\usepackage{listings}
\usepackage{color} 
\usepackage{xcolor}
\definecolor{codegreen}{rgb}{0,0.6,0}
\definecolor{codegray}{rgb}{0.5,0.5,0.5}
\definecolor{codepurple}{rgb}{0.58,0,0.82}
\definecolor{backcolour}{rgb}{0.95,0.95,0.92}
\lstdefinestyle{mystyle}{
    backgroundcolor=\color{backcolour},   
    commentstyle=\color{codegreen},
    keywordstyle=\color{magenta},
    numberstyle=\tiny\color{codegray},
    stringstyle=\color{codepurple},
    basicstyle=\ttfamily\footnotesize,
    breakatwhitespace=false,         
    breaklines=true,                 
    captionpos=b,                    
    keepspaces=true,                 
    numbers=left,                    
    numbersep=6pt,                  
    showspaces=false,                
    showstringspaces=false,
    showtabs=false,                  
    tabsize=4
}

\usepackage{multirow}  
\usepackage{float}
\usepackage{tcolorbox}
\usepackage[numbered]{matlab-prettifier}
\usepackage{diagbox}
\usepackage{cleveref}
\usepackage{diagbox}
\usepackage{algorithm}
\usepackage{algpseudocode}
\usepackage{caption}
\usepackage{subcaption}
\usepackage{textcomp,multirow}
\theoremstyle{plain}
\newtheorem{thm}{Theorem}[section]
\newtheorem*{thm*}{thm}
\newtheorem{conj}{Conjecture}[section]

\numberwithin{equation}{section}

\newtheorem{lem}{Lemma}[section]
\newtheorem{prop}{Proposition}[section]

\newtheorem{rem}{Remark}[section]

\newtheorem{defn}{Definition}[section]

\theoremstyle{definition}
\newcounter {own}
\def\theown {\thesection  .\arabic{own}}

\newenvironment{pf}[1][]{%
 \vskip 2mm
 \noindent
 \ifthenelse{\equal{#1}{}}%
  {{\slshape Proof. }}%
  {{\slshape #1.} }%
 }%
{\qed\bigskip}
\graphicspath{ {./images/} }

\DeclareMathOperator*{\supp}{\supp}
\DeclareMathOperator*{\rank}{rank}

\newcounter{alphabet}

\newcommand{\R}{\mathbb{R}}

\newcommand{\Z}{\mathbb{Z}}

\newcommand{\qtq}[1]{\quad\text{#1}\quad}

\newtheorem{theorem}{Theorem}[section]

\theoremstyle{definition}

\theoremstyle{remark}

\newcommand{\ds}{\displaystyle}

\begin{document}
\title[Gabor Frames for a Rational Window with Symmetric Poles]{Gabor Frame for a Rational Window with Symmetric Poles}
\thanks{Indian Institute of Technology (IIT) Bombay, Mumbai, India, 400076\\
Email: riya74012@gmail.com}
\author{Riya Ghosh}


\begin{abstract}
We investigate the Gabor frame problem for a rational window
$$g(x)=
\frac{x}{(x^2+1)(x^2+4)}.$$
This window falls outside the classes covered by existing frame set characterizations for rational Herglotz functions and ratios of exponential polynomials.  For rational densities ($\alpha\beta = p/q$), we leverage the Zak transform to map the Gabor frame condition directly to a finite-dimensional equivalence criterion, which is completely determined by the rank of an associated polynomial matrix. Applying this framework, we establish new frame results for every rational density $$\alpha\beta=\frac{p}{q}\in \left(0,\frac{1}{2}\right)\cup\left(\frac{1}{2},\frac{2}{3}\right)\cup \left(\frac{2}{3},\frac{5}{7}\right].$$
\end{abstract}
\date{\today}
\maketitle
\pagestyle{myheadings}
\markboth{Riya Ghosh}{Gabor Frames for a Rational Window with Symmetric Poles}
\section{Introduction and Preliminaries}
Let $g\in L^2(\mathbb{R})$ be a non-zero window function with lattice parameters $\alpha, \beta > 0$. The corresponding Gabor system  $\mathcal{G}(g,\alpha, \beta)=\{e^{2\pi i\beta m\cdot}g(\cdot-\alpha k): k, m\in\mathbb{Z}\}$ is called a (Gabor) frame for $L^2(\mathbb{R})$ if there exist two positive constants $A$, $B$ such that 
\begin{equation}\label{Gaborframe}
A\|f\|^2\leq\displaystyle\sum_{m,k\in\mathbb{Z}}|\langle f,e^{2\pi i\beta m\cdot} g(\cdot-\alpha k)\rangle|^2\leq B\|f\|^2,
\end{equation}
for every $f\in L^2(\mathbb{R})$. The constants $A$ and $B$ are called frame bounds. One of the fundamental problems in Gabor analysis is to determine the values of $\alpha, \beta>0$ such that $\mathcal{G}(g,\alpha, \beta)$ is a frame for $L^2(\mathbb{R})$. The set of all such lattice parameters is referred to as the frame set of $g$ and is given by
$$\mathcal{F}(g)= \left\{(\alpha, \beta) \in\mathbb{R}^2_+ : \mathcal{G}(g, \alpha,\beta)~ \text{is a frame}\right\} .$$

In this paper, we use the following version of the Fourier transform: 
$$\widehat f(w)= \int_{-\infty}^{\infty} f(x) e^{-2\pi \mathrm{i}wx}dx,~ w\in\mathbb{R}.$$
Feichtinger and Kaiblinger \cite{HGF1} proved that $\mathcal{F}(g)$ is an open subset of $\mathbb{R}^2_+$ for a window $g$ in Feichtinger algebra. The fundamental density theorem asserts that 
$$\mathcal{F}(g)\subseteq \{(\alpha, \beta) \in\mathbb{R}^2_+: \alpha\beta\le 1\}$$ 
(see \cite{time, hedtg}). In addition, if $g$ is in the Feichtinger algebra, Balian-Low theorem states that $\mathcal{F}(g)\subseteq \{(\alpha, \beta) \in\mathbb{R}^2_+: \alpha\beta< 1\}$ \cite{balian, rdbalian}. For a more comprehensive discussion on Gabor analysis, we refer to \cite{ole, time}. 

The frame set is completely characterized only for a few windows including the Gaussian $e^{-\pi x^2}$ \cite{fblyu, dtsibf2}, the hyperbolic secant \cite{hsyg}, the one-sided exponential $e^{-x}\chi_{[0,\infty)}(x)$ \cite{whfb}, the two-sided exponential  $e^{-|x|}$ \cite{tgfcd}, the characteristic function $\chi_{[0,c)},~c>0$ \cite{abc, when}, the totally positive functions of finite type $\geq 2$ or of Gaussian type \cite{stsis, duke}, the Haar function \cite{Haar} and the Herglotz functions \cite{Yurii:rational:2023}, whose poles lie entirely in one half-plane. Using a sufficient condition for Gabor frames that connects sampling theory in shift-invariant spaces with Gabor analysis, the authors in \cite{ghosh2025gabor} present a frame region that is close to the frame set of a totally positive function, along with explicit frame bounds. For a complete characterization of Gabor frames for a given window, we refer to \cite{feichtinger2012advances, feichtinger2012gabor,grochenig2014mystery}.

In this paper, we investigate the Gabor frame properties of a rational window
\begin{align}\label{window}
    g(x)=\frac{x}{(x^2+1)(x^2+4)}
\end{align}
The poles of $g$ occur in symmetric pairs $\pm ik$, $k=1,2$. Moreover, $g\in \mathcal{S}_0(\mathbb{R})\cup W(L^\infty,\ell^1)$. The partial-fraction decomposition of $g$ plays a central role in our analysis. More precisely, writing
\begin{align}\label{defn w_k}
    \omega_k=k,\qtq{and} \omega_{2+k}=-\omega_k,\qtq{for} 1\le k\le 2,
\end{align}
we have
\begin{align}\label{rational}
g(t)=\sum_{k=1}^{4}\frac{a_k}{t-i\omega_k},
\end{align}
where,
\begin{align}\label{defn a_k}
a_1=a_3=\frac{1}{6},\qtq{and}a_2=a_4=-\frac{1}{6}.
\end{align}
Here, 
$\sum_{k=1}^{4} a_k=0$ with $\sum_{k=1}^{4}\omega_k=0$. 

The frame-set problem for these windows lies outside the scope of several existing complete characterizations. Belov \textit{et al.} \cite{Yurii:rational:2023} obtained a complete characterization for rational Herglotz functions whose poles lie entirely in a common half-plane. More recently, Ulanovskii and Zlotnikov \cite{ulanovskii2025sampling} studied generators represented as ratios of exponential polynomials with exponential or Gaussian decay. In contrast, the windows $g$ defined in \eqref{window} are purely algebraic rational functions with polynomial decay $O(|x|^{-3})$ and symmetric poles at $\pm ik$. Thus, they are neither Herglotz functions nor ratios of exponential polynomials, and therefore lie outside the scope of these general frame set characterizations.

There is, however, an important distinction between the irrational and rational densities $\alpha\beta$. In Section \ref{section 3}, we prove that
$$m_0(\xi)=\sum_{k=1}^{4}a_k e^{2\pi \xi\omega_k}\neq0,\quad \xi>0,$$
Consequently, Theorem 1.2 in \cite{Yurii:rational:2023} guarantees that the Gabor system $\mathcal{G}(g,\alpha,\beta)$ forms a frame whenever $\alpha\beta<1$ is irrational. Thus, for $g$, the genuinely unresolved part of the frame set problem is the rational density case. The main purpose of this paper is to investigate this remaining case by exploiting the explicit algebraic structure of the Zak transform.

Recently, Semenov \cite{semenov2025} established a universal frame result for a broad class of rational functions. For every $M\in\mathbb N$ and $\varepsilon>0$, he constructed a nonuniform frequency set $\Lambda\subset\mathbb R$ with $D(\Lambda)\le 1+\varepsilon$ such that $\mathcal G(g,\Lambda\times\mathbb Z)$ is a frame for every rational window $g$ in his class of degree at most $M$. The windows $g$ considered here belong to this class and have degree $4$. However, Semenov's result concerns a specially constructed nonuniform frequency set, whereas the present paper studies the classical rectangular lattice $\alpha\mathbb Z\times\beta\mathbb Z$. Thus, his result does not determine the rectangular-lattice frame set considered here.

Since $g$ is odd, the obstruction of Lyubarskii and Nes \cite{lyubarskii2013gabor} gives
    $$\mathcal{G}(g,\alpha,\beta)\text{ does not form a frame in }L^2(\R)\text{ for }\alpha\beta=\frac{N}{N+1},~N=1,2,\dots$$

This obstruction motivates the search for positive rational frame results for the odd window $g$. To treat rational densities, we reduce the Gabor frame condition directly to a finite-dimensional polynomial matrix. Although this formulation is rank-equivalent to the classical Zibulski--Zeevi matrix \cite{zibulski1997analysis}, our explicit polynomial construction exposes the algebraic structure of the Zak transform equations more transparently. In particular, the kernel equations become a finite-dimensional polynomial system whose structure allows us to eliminate the coefficients directly. For $g$, this leads to different kernel-elimination mechanisms in different density regimes. In particular, the combinatorial structure of the matrix $\Phi$ changes at the densities $3/5$ and $5/7$, which naturally leads to separate arguments for the intervals
$$0<\alpha\beta=\frac{p}{q}<\frac{1}{2},\quad \frac{1}{2}<\alpha\beta=\frac{p}{q}<\frac{3}{5},\quad\frac{3}{5}<\alpha\beta=\frac{p}{q}<\frac{2}{3},\qtq{and} \frac{2}{3}<\alpha\beta=\frac{p}{q}<\frac57.$$
Our main result in this range is the following.
\begin{thm}\label{p/qin(1/2,2/3)}
The Gabor system $\mathcal{G}(g, \alpha, \beta)$ forms a frame for $L^2(\mathbb{R})$ for every reduced rational density 
$$\alpha\beta = \frac{p}{q} \in \left(0,\frac{1}{2}\right)\cup\left(\frac{1}{2}, \frac{2}{3}\right)\cup \left(\frac{2}{3}, \frac{5}{7}\right].$$
\end{thm}

\begin{figure}[H]
    \centering
\includegraphics[width=0.8\textwidth,height=13cm]{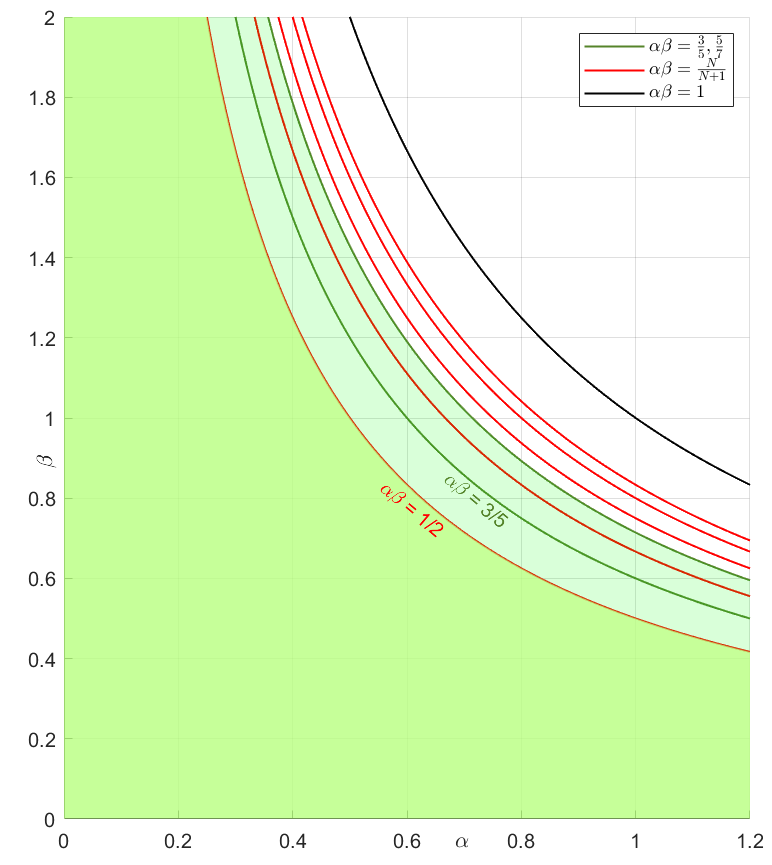}
    \caption{This plot illustrates that the Gabor system $\mathcal{G}(g, \alpha, \beta)$ fails to form a frame when $\alpha\beta = \frac{N}{N+1}$ (red) and $\alpha\beta=1$ (black). The green regions indicate the rational densities in Theorem \ref{p/qin(1/2,2/3)}, 
    $\alpha\beta=\frac{p}{q}\in \left(0,\frac{1}{2}\right)\cup\left(\frac{1}{2}, \frac{2}{3}\right)\cup \left(\frac{2}{3}, \frac{5}{7}\right],$
together with irrational densities $0<\alpha\beta<1$ in \cite{Yurii:rational:2023}.}
    \label{Conj Fig}
\end{figure}  

Our main contributions are summarized as follows.
\begin{enumerate}
\item For rational densities $\alpha\beta = p/q$, we leverage the Zak transform to map the infinite-dimensional Gabor frame condition directly to a finite-dimensional equivalence criterion. This fundamentally reduces the frame verification problem to evaluating the full-rank condition of an explicitly constructed polynomial matrix.

\item We prove that the Gabor system $\mathcal{G}(g,\alpha,\beta)$ forms a frame for every reduced rational density 
$$\alpha\beta=\frac{p}{q}\in \left(0,\frac{1}{2}\right)\cup\left(\frac12,\frac23\right)\cup \left(\frac23,\frac57\right].$$
\end{enumerate}

The paper is organized as follows. In Section \ref{section 2} we introduce the generating polynomial associated with the partial-fraction representation of $g$, derive its reflection symmetry, and obtain an explicit Zak transform formula. In Section \ref{section 3} we use the Zibulski–Zeevi representation to reduce the rational-density Gabor frame problem to the full-column-rank condition of a finite polynomial matrix and establish a symmetry reduction in the frequency variable. Section \ref{section 4} proves Theorem \ref{p/qin(1/2,2/3)} by analyzing the resulting finite systems in the different rational-density regimes. The auxiliary monotonicity and positivity estimates used in these arguments are collected in the Appendix.

\section{Generating Polynomial and Zak Transform}\label{section 2} 
For $s= 0,1,2,3$, we consider the function
\begin{align}\label{E : def m_S}
m_{s}(\xi) = \sum_{k = 1}^{4} a_{k} A_{k, s} e^{2 \pi \xi \omega_k}, \quad \xi \in \R
\end{align}
where the coefficients $a_k$ are defined in \eqref{defn a_k},
and the terms $A_{k, s}$ are defined as
\begin{align}\label{Defn Aks}
    A_{k, s}=
    (- 1)^s \ds\sum\limits_{\substack{1\le j_1 < \cdots < j_s\le 4\\ j_l \ne k}} e^{2 \pi(\omega_{j_1} + \cdots + \omega_{j_s})}.
\end{align}
We can easily verify that
\begin{align}\label{boundary m relation}
    m_0(0)=\sum_{k = 1}^{4} a_{k}=0,\qtq{and} m_{3}(\xi)=\sum_{k = 1}^{4} a_{k} A_{k, 3} e^{2 \pi \xi \omega_k}=-m_0(\xi-1).
\end{align} 
Let $\widetilde{A}_{s}=\sum\limits_{1\le j_1<\cdots<j_s\le 4} e^{2 \pi(\omega_{j_1}+\cdots+\omega_{j_s})}$ with $\widetilde{A}_{0}=1$. It is easy to check that
\begin{align}\label{Aks and tilde As}
    A_{ks} = e^{2\pi \omega_k}\, A_{k,s-1} + (-1)^s \widetilde{A}_{s}.
\end{align}
Using \eqref{Aks and tilde As}, we obtain
\begin{align}\label{ms recurrence}
    m_{s}(\xi)=m_{s-1}(\xi+1)+(-1)^s\widetilde{A}_{s}\,m_{0}(\xi).
\end{align}
We observe the identity
\begin{align*}
   \frac{\sum\limits_{s=0}^{3} m_{s}(\xi)\,z^s}{\prod\limits_{k=1}^{4}(1 - z e^{2 \pi \omega_k})}= \sum_{k=1}^{4} \frac{a_{k}e^{2 \pi\xi \omega_k}}{1 - z e^{2 \pi \omega_k}}.
\end{align*}
Let us define
\begin{align}\label{P_Ndefinition}
\mathcal{P}(z\, ;\xi)=\sum_{s=0}^{3} m_{s}(\xi)\,z^s=R(z)\, Q(z\,;\xi),
\end{align}
where
$$R(z)=\prod_{k=1}^{4} (1 - z e^{2 \pi \omega_k})\qtq{and}Q(z\,;\xi)=\sum_{k=1}^{4} \frac{a_{k}e^{2 \pi\xi \omega_k}}{1 - z e^{2 \pi \omega_k}}.$$
Here, $\mathcal{P}(z\, ;\xi)$ is a polynomial in $z$ of degree $3$ for each fixed $\xi\in [0,1)$. Setting $\gamma_k=e^{2\pi \omega_k}$ yields $\gamma_{2+k}=\gamma_k^{-1}$ allowing
    $$R(z)=\prod_{k=1}^2\left(1-z\gamma_k\right)\left(1-z\gamma_k^{-1}\right),$$
    which explicitly satisfies the self-reciprocal identity
    \begin{align}\label{reciprocal_R}
        z^{4} R(1/z)=\prod_{k=1}^2\left(z-\gamma_k\right)\left(z-\gamma_k^{-1}\right)=R(z).
    \end{align}

\begin{lem} \label{poly_reciprocity}
For $\xi \in [0,1]$, we have
\begin{equation}\label{eq:poly_reciprocity}
z^{3}\mathcal{P}\left(\frac{1}{z};\xi\right) = -\,\mathcal{P}(z;1-\xi).
\end{equation}
\end{lem}

\begin{pf}
Let $\gamma_k=e^{2\pi \omega_k}$. Now
\begin{align}
Q\left(\frac{1}{z};\xi\right) &= \sum_{k=1}^{4} \frac{a_k e^{2\pi \xi \omega_k}}{1 - \frac{1}{z}\gamma_k} = \sum_{k=1}^{4} \frac{z a_k e^{2\pi \xi \omega_k}}{z - \gamma_k}\nonumber\\
&= \sum_{k=1}^2 \frac{z a_k e^{2\pi \xi \omega_k}}{z - \gamma_k} + \sum_{k=1}^2 \frac{z a_k e^{-2\pi \xi \omega_k}}{z - \gamma_k^{-1}}\nonumber\\
&= \sum_{k=1}^2 \frac{z a_k \gamma_k^{-1} e^{2\pi \xi \omega_k}}{z\gamma_k^{-1} - 1} + \sum_{k=1}^2 \frac{z a_k \gamma_k e^{-2\pi \xi \omega_k}}{z\gamma_k - 1}\nonumber\\
&= -\sum_{k=1}^2 \frac{z a_k e^{-2\pi(1-\xi)\omega_k}}{1 - z\gamma_k^{-1}} - \sum_{k=1}^2 \frac{z a_k e^{2\pi(1-\xi)\omega_k}}{1 - z\gamma_k}.\nonumber
\end{align}
By definition, 
$$Q(z; 1-\xi) = \sum_{k=1}^2 \frac{a_k e^{2\pi(1-\xi)\omega_k}}{1 - z\gamma_k} + \sum_{k=1}^2 \frac{a_k e^{-2\pi(1-\xi)\omega_k}}{1 - z\gamma_k^{-1}}.$$
Therefore, 
\begin{align}\label{shift_reciprocal_Q}
    Q\left(\frac{1}{z};\xi\right) = - z Q(z; 1-\xi).
\end{align}
Combining \eqref{reciprocal_R} and \eqref{shift_reciprocal_Q}, we obtain
\begin{align*}
\mathcal{P}\left(\frac{1}{z};\xi\right) = R\left(\frac{1}{z}\right)Q\left(\frac{1}{z};\xi\right)& = \left(z^{-4}R(z)\right) \left(-z Q(z;1-\xi)\right)\\
&= -z^{-3}\mathcal{P}(z;1-\xi).
\end{align*}
Multiplying through by $z^{3}$ completes the proof.
\end{pf}

\begin{rem}\label{xi_1/2_rem}
Evaluating Lemma \ref{poly_reciprocity} at $\xi = 1/2$ directly establishes that $\mathcal{P}(z\,;1/2)$ is anti self-reciprocal.
\end{rem}

\begin{defn}
    The Zak transform of a function $g$ is defined by
\begin{align}\label{Zak}
\mathcal{Z}_g(t,\xi)=\sum_{m\in\mathbb{Z}}g(t-m)e^{2\pi im\xi},\quad (t,\xi)\in \R^2.
\end{align}
\end{defn}

\begin{prop}\label{P: zak and Q}
If $g$ is defined in \eqref{rational}, then
    $$\mathcal{Z}_{g}(t\,;\xi)=-2\pi i \, e^{2\pi i\xi t}\,Q(z\,;\xi),\qtq{for} 0<\xi<1,$$
    where $z=e^{2\pi it}$ and $Q(z\,;\xi)$ is defined in \eqref{P_Ndefinition}.
\end{prop}

\begin{proof}
    For $\nu\notin \Z$ and $0<\xi<1$, we utilize the standard identity 
    \begin{align}\label{infinite_sum}
        \sum_{m\in\Z}\frac{e^{2\pi i m\xi}}{\nu-m}=\frac{-2\pi i \,e^{2\pi i\nu\xi}}{1-e^{2\pi i\nu}}.
    \end{align}
Let $z=e^{2\pi it}$. From the definition of the Zak transform of $g$, we have
\begin{align*}
    \mathcal{Z}_{g}(t\,;\xi)&=\sum_{k=1}^{4} a_k\sum_{m\in \Z}\frac{e^{2\pi im\xi}}{t-m-i\omega_k}=-2\pi i\sum_{k=1}^{4} a_k \frac{e^{2\pi i\xi(t-i\omega_k)}}{1-e^{2\pi i(t-i\omega_k)}}=-2\pi i \, e^{2\pi i\xi t}\, Q(z\,;\xi).
\end{align*}
\end{proof}

\begin{rem}
    For the case $\xi=0$, we use the following identities. $$\sum_{m\in\mathbb Z}\frac1{z-m}=
\pi\cot(\pi z),\quad \cot(\pi u)=i\frac{e^{2\pi i u}+1}{e^{2\pi i u}-1}, \qtq{and} \sum_{k=1}^{4} a_k=0.$$
Hence,
    \begin{align*}
        \mathcal Z_g(t,0)=
-\pi i
\sum_{k=1}^{4}
a_k
\frac{1+z e^{2\pi\omega_k}}
{1-z e^{2\pi\omega_k}}=
-\pi i\sum_{k=1}^{4}
a_k
\left(1+\frac{2z e^{2\pi\omega_k}}
{1-z e^{2\pi\omega_k}}\right) =-2\pi i\, Q(z\,;0).
    \end{align*}
\end{rem}

\section{Frame Property for rational window}\label{section 3}
Now 
\begin{align}
    m_0(\xi)=\sum_{k=1}^{4}a_k e^{2\pi \xi \omega_k}&=
2\displaystyle\sum_{k=1}^{2} a_k \cosh(2 \pi \xi k)\nonumber\\
&=\frac{1}{3}(\cosh(2\pi\xi)-\cosh(4\pi\xi))<0,~\text{for }\xi>0.
\end{align}
Since $m_0(\xi)\ne 0$ for $\xi>0$ and $\omega_k\ne \omega_l$, then the Gabor system $\mathcal{G}(g, \alpha,\beta)$ is a frame for $L^2(\R)$ whenever $\alpha\beta<1$ is irrational (see \cite{Yurii:rational:2023}). Furthermore, Daubechies \cite{Daubechies} established that $\left\{(\alpha,\beta):~\alpha\beta\le \frac{1}{4}\right\}$ is safely contained within $\mathcal{F}(g)$. In this paper, therefore, we focus primarily on rational lattices.

It is well-known that if $\mathcal{G}(g, \alpha, \beta)$ forms a Gabor frame, then scaling the argument of $g$ by $ 1/\beta $ gives another frame
$$\mathcal{G}\left(\beta^{-1/2}g\left(\frac{\cdot}{\beta}\right), \alpha\beta, 1\right)$$
  with the same bounds.

Let us define $g_\beta(t)=g(t/\beta)$.
Let $\mathbf M(t\,;\xi)$ be a $q\times p$ Zibulski-Zeevi matrix \cite{zibulski1997analysis} for $g_\beta$ with entries
\begin{align}\label{M_matrix}
    \mathbf M_{sr}(t\,;\xi)=\mathcal Z_{g_\beta}(t-\alpha\beta s;\xi+r/p)
\end{align}
for $(t\,;\xi)\in [0,1)\times[0,1/p)$. We define an associated $q \times p$ matrix $\mathbf B(z;\xi)$ on $(z\,;\xi)\in \mathbb{T}\times[0,1/p)$ with entries 
$$(\mathbf B (z;\xi))_{sr} = e^{-2\pi i\frac{rs}{q}}Q\left(ze^{-2\pi i\alpha\beta s};\xi+\frac{r}{p}\right),$$ 
where
$$Q_\beta(z;\xi)=\sum_{k=1}^4 \frac{a_k e^{2\pi \xi \beta \omega_k}}{1-ze^{2\pi \beta \omega_k}}$$
for $0\le s\le q-1$ and $0\le r\le p-1$.
From Proposition \ref{P: zak and Q}, 
$$\mathbf M_{sr}(t\,;\xi)=
\Bigl(
-2\pi i\beta\,
e^{2\pi i\xi(t-\alpha\beta s)}
e^{2\pi i(r/p)t}
\Bigr)
\,(\mathbf B (z\,;\xi))_{sr},\quad z=e^{2\pi it}.
$$
Therefore,
$$\mathbf M(t\,;\xi)=
\operatorname{diag}(D_s(t\,;\xi))\,
\mathbf{B}(z\,;\xi)\,
\operatorname{diag}(E_r(t)),
$$
where
$$D_s(t\,;\xi)=-2\pi i\beta\,e^{2\pi i\xi(t-\alpha\beta s)}
\qtq{and} E_r(t)=
e^{2\pi i(r/p)t}.$$
Since both diagonal matrices are invertible 
$$\rank \mathbf{M}(t\,;\xi)=\rank \mathbf{B}(z\,;\xi).$$

Now define
$$R_\beta(z)=\prod_{k=1}^{2}
(1-ze^{2\pi\beta k})
(1-ze^{-2\pi\beta k}),$$
and
$$\mathcal{P}(z;\xi)=
R_\beta(z)Q_\beta(z;\xi)=\sum_{l=0}^{3}m_{l,\beta}(\xi)z^l.$$

Finally, we introduce a $q\times p$ matrix $ \Phi(z\,; \xi)$
\begin{align}\label{final Phi_alpha matrix}
    &= \begin{bmatrix}
    \mathcal{P}(z\,;\xi)& \mathcal{P}\left(z\,;\xi+\frac{1}{p}\right)& \cdots& \mathcal{P}\left(z\,;\xi+\frac{p-1}{p}\right)\\[0.25em]
   \mathcal{P}(ze^{-2\pi i\alpha\beta}\,;\xi)& e^{-\frac{2\pi i}{q}}\mathcal{P}\left(ze^{-2\pi i\alpha\beta}\,;\xi+\frac{1}{p}\right)& \cdots&e^{-2\pi i\frac{p-1}{q}}\mathcal{P}\left(ze^{-2\pi i\alpha\beta}\,;\xi+\frac{p-1}{p}\right)\\[0.25em]
    \vdots&\vdots&\cdots&\vdots\\[0.25em]
    \mathcal{P}(ze^{-2\pi i\alpha\beta(q-1)}\,;\xi)& e^{\frac{2\pi i}{q}}\mathcal{P}\left(ze^{-2\pi i\alpha\beta(q-1)}\,;\xi+\frac{1}{p}\right)& \cdots&e^{2\pi i\frac{p-1}{q}}\mathcal{P}\left(ze^{-2\pi i\alpha\beta(q-1)}\,;\xi+\frac{p-1}{p}\right)
\end{bmatrix}.
\end{align}
Consequently, from \eqref{P_Ndefinition}, we have
\begin{align*}
(\Phi(z\,;\xi))_{sr}&=e^{-2 \pi i \frac{r}{q}s}\,\mathcal{P}\left(ze^{-2 \pi i \alpha\beta s}\, ;\xi+\frac{r}{p}\right),\\
&=e^{-2\pi i\frac{r}{q}s}
R_\beta(ze^{-2\pi i\alpha \beta s})\,Q_\beta\left(
ze^{-2\pi i\alpha\beta s}\,;
\xi+\frac r p
\right) 
\end{align*}
for $z=e^{2\pi it}$. Since the zeros of $R_\beta$ are $e^{\pm2\pi\beta k}$,
none lie on $\mathbb T$. Hence for $|z|=1$,
$$R_\beta(z_s)\neq0,\quad z_s=ze^{-2\pi i\alpha\beta s}$$
so
$$\operatorname{rank}\mathbf{M}(t,\xi)=\operatorname{rank}
\Phi(z;\xi).$$

The following theorem establishes that the frame property for the Gabor system generated by $g$ is equivalent to the full-rank condition of the Zibulski--Zeevi matrix \eqref{M_matrix} and its equivalent matrix representations \eqref{final Phi_alpha matrix}.
\begin{theorem}\label{rank equiv}
    Let $g$ be a window function in \eqref{window}. Then the Gabor system $\mathcal{G}(g,\alpha,\beta)$ forms a frame for $\alpha\beta=\frac{p}{q}$ if and only if 
    $$\rank \mathbf{M}(t\,;\xi)=\rank \mathbf{B} (z\,;\xi)=\rank \Phi(z\,;\xi)=p,$$
    for $(z\,,\xi)\in \mathbb{T}\times \left[0,1/p\right)$.
\end{theorem}

The critical density $\alpha\beta=1$ is exceptional. In this case, we obtain from \eqref{final Phi_alpha matrix} that 
$$\Phi_{1}(z\,;\xi)=\mathcal{P}(z\,;\xi).$$ 
Remark \ref{xi_1/2_rem} shows that
$$\mathcal{P}(1\,;1/2)=0.$$
Consequently, therefore $\mathcal{G}(g, 1, 1)$ is not a frame.

\begin{rem}
Although $\Phi(z,;\xi)$ is rank-equivalent to the classical Zibulski--Zeevi matrix, it is adapted to the rational window $g$ through its partial-fraction expansion. The resulting polynomial representation reformulates the Gabor frame condition as the invertibility of a finite-dimensional structured matrix whose entries are explicitly determined by the coefficients of the exponential polynomials $\mathcal P(z\,;\xi)$. As shown in Section \ref{section 4}, this formulation naturally reveals sparse staircase structures that can be analyzed by determinant identities, Schur complements, and finite recurrences. These structural properties provide an effective algebraic framework for proving the invertibility of the polynomial matrix at rational densities considered in this paper.
\end{rem}

To explicitly verify the full-rank condition of Theorem \ref{rank equiv}, we must analyze the algebraic properties of its associated matrix $\Phi(z\,;\xi)$. By leveraging the polynomial reflection identities in \eqref{eq:poly_reciprocity}, the following lemma demonstrates a fundamental structural symmetry in the frequency parameter $\xi$. Consequently, it is sufficient to analyze the $\Phi(z\,;\xi)$ matrix only for $\xi\in [0, \frac{1}{2p}]$.

\begin{lem}\label{symmetry_singular_values}
Let $\alpha\beta=p/q$ and let $\Phi(z;\xi)$ be the matrix associated
with the window $g$. Then, for every $z\in\mathbb T$ and
$\xi\in[0,1/p)$, there exist permutation matrices $R$ and $P$ and a diagonal unitary matrix $U(z)$ such that
\begin{equation}\label{Phi_reflection}
\Phi\left(z;\frac1p-\xi\right)
=
U(z)\,R\,\Phi(z^{-1};\xi)\,P .
\end{equation}
Consequently,
$$
\operatorname{rank}\Phi\left(z;\frac1p-\xi\right)
=
\operatorname{rank}\Phi(z^{-1};\xi).
$$
In particular, it is enough to verify the full-rank condition for $\xi\in\left[0,\frac1{2p}\right].$
\end{lem}

\begin{proof}
For $0\le s\le q-1$ and $0\le r\le p-1$, put
$$
z_s=ze^{-2\pi i ps/q}.
$$
By definition,
\begin{equation}\label{Phi_entry_sym}
\left(
\Phi\left(z;\frac1p-\xi\right)
\right)_{sr}
=
e^{-2\pi i rs/q}
P_\beta\left(
z_s;\frac1p-\xi+\frac rp
\right).
\end{equation}

For each $r$, define $r'=p-1-r.$
Then
\begin{align*}
\frac1p-\xi+\frac rp
=
\frac{r+1}{p}-\xi&=
1-\left(\xi+\frac{p-r-1}{p}\right)\\
&=
1-\left(\xi+\frac{r'}p\right).
\end{align*}
From Lemma \ref{poly_reciprocity}, we have
$$
P_\beta\left(
z_s;\frac1p-\xi+\frac rp
\right)
=
-z_s^3
P_\beta\left(
z_s^{-1};\xi+\frac{r'}p
\right).
$$
Now \eqref{Phi_entry_sym} gives
\begin{equation}\label{Phi_reflected_entry}
\left(
\Phi\left(z;\frac1p-\xi\right)
\right)_{sr}
=
-z_s^3e^{-2\pi i rs/q}
P_\beta\left(
z_s^{-1};\xi+\frac{r'}p
\right).
\end{equation}

We now reflect the row index as well. Define
$$
s'=-s\pmod q.
$$
Thus
$$
s'=
\begin{cases}
0,&s=0,\\
q-s,&1\le s\le q-1.
\end{cases}
$$
Since
$$
z_s^{-1}
=
z^{-1}e^{2\pi i ps/q}
=
z^{-1}e^{-2\pi i ps'/q},
$$
the polynomial appearing on the right-hand side of
\eqref{Phi_reflected_entry} is precisely the one occurring in the
$(s',r')$ entry of $\Phi(z^{-1};\xi)$. Moreover,
$$
e^{-2\pi i r's'/q}
=
e^{2\pi i r's/q},\qtq{and} r+r'=p-1,
$$
we have
\begin{align*}
e^{-2\pi i rs/q}
=e^{-2\pi i(p-1)s/q}
e^{-2\pi i r's'/q}.
\end{align*}
Therefore,
\begin{align}
\left(
\Phi\left(z;\frac1p-\xi\right)
\right)_{sr}
=&
u_s(z)\,
e^{-2\pi i r's'/q}
P_\beta\left(
z^{-1}e^{-2\pi i ps'/q};
\xi+\frac{r'}p
\right),
\label{entry_final}
\end{align}
where
$$
u_s(z)
=
-z_s^3e^{-2\pi i(p-1)s/q}.
$$
Since $|u_s(z)|=1$,
$$
U(z)=
\operatorname{diag}
\big(
u_0(z),u_1(z),\ldots,u_{q-1}(z)
\big)
$$
is unitary. Let $R$ be the $q\times q$ permutation matrix corresponding to
$$
s\longmapsto -s\pmod q,
$$
and let $P$ be the $p\times p$ permutation matrix corresponding to
$$
r\longmapsto p-1-r.
$$
Then \eqref{entry_final} is exactly the matrix identity
$$
\Phi\left(z;\frac1p-\xi\right)
=
U(z)\,R\,\Phi(z^{-1};\xi)\,P.
$$
Since $U(z)$, $R$, and $P$ are invertible,
$$
\operatorname{rank}
\Phi\left(z;\frac1p-\xi\right)
=
\operatorname{rank}\Phi(z^{-1};\xi).
$$
Therefore, it is sufficient to establish the full-rank condition for
$$(z,\xi)\in\mathbb T\times\left[0,\frac1{2p}\right].$$
\end{proof}

\section{Proof of Theorems}\label{section 4}

To prove Theorem \ref{p/qin(1/2,2/3)}, we first determine the coefficients of the polynomial $\mathcal{P}(z\,;\xi)$. Using \eqref{boundary m relation} and \eqref{ms recurrence}, we obtain 
\begin{align*}
    m_{0,\beta}(\xi)&= \frac{1}{6}\left( e^{2\pi\beta\xi}- e^{4\pi\beta\xi} +e^{-2\pi\beta\xi} - e^{-4\pi\beta\xi}\right)=\frac{1}{3}(\cosh(2\pi\beta\xi)-\cosh(4\pi\beta\xi))\\
    m_{1,\beta}(\xi)&= m_{0,\beta}(\xi+1)- (e^{2\pi\beta} + e^{4\pi\beta} + e^{-2\pi\beta} + e^{-4\pi\beta}) m_{0,\beta}(\xi)\\
    &=m_{0,\beta}(\xi+1)- 2(\cosh(2\pi\beta)+\cosh(4\pi\beta)) m_{0,\beta}(\xi)\\
    m_{2,\beta}(\xi)
&=m_{1,\beta}(\xi+1)+2(\cosh(6\pi\beta)+\cosh(2\pi\beta)+1)m_{0,\beta}(\xi)\\
    m_{3,\beta}(\xi)& = -m_{0,\beta}(\xi-1).
\end{align*}
It is straightforward to verify that $m_{0,\beta}(\xi)$ is an even negative function for $\xi > 0$ and $m_{1,\beta}(\xi) < 0$ for $\xi \in [0, 1)$. Similarly, $m_{2,\beta}(\xi) > 0$ for $\xi \in (0, 1]$ and $m_{3,\beta}(\xi) > 0$ for $\xi \in [0, 1)$. Furthermore, Lemma \ref{poly_reciprocity} provides
\begin{equation}\label{reflection-symmetry}
m_{3,\beta}(\xi)=-m_{0,\beta}(1-\xi) \quad \text{and} \quad m_{2,\beta}(\xi)=-m_{1,\beta}(1-\xi).
\end{equation}
Recall that the matrix $\Phi(z; \xi)$ defined in \eqref{final Phi_alpha matrix} is a $q \times p$ matrix. By Lemma \ref{symmetry_singular_values}, it suffices to prove that
   \begin{align}\label{kerPhi_g1}
       \ker\Phi(z; \xi) = \{\mathbf{0}\}, \text{ for all }(z, \xi) \in \mathbb{T} \times \left[0, \frac{1}{2p}\right].
   \end{align}
Suppose $\mathbf{c} = [c_0, c_1, \dots, c_{p-1}]^T \in \mathbb{C}^{p}$ is a vector in the kernel, such that $\Phi(z\,; \xi)\mathbf{c} = \mathbf{0}$. Therefore,
    $$\sum_{r=0}^{p-1} c_r e^{-2\pi i \frac{rs}{q}} \mathcal{P}(z_s; \xi_r) = 0, \quad \text{for } s = 0, 1, \dots, q-1.$$
Using \eqref{P_Ndefinition} to expand the polynomial and substituting $z_s = z e^{-2\pi i \frac{ps}{q}}$, we can rewrite this as
\begin{align}\label{system_of_eqns}
     \sum_{l=0}^3 z^l \left( \sum_{r=0}^{p-1} c_r e^{-2\pi i s \left( \frac{r+pl}{q} \right)} m_{l,\beta}(\xi_r) \right) = 0. 
\end{align}
Define
\begin{align}\label{D_k equation}
     D_k(z\,;\xi) = \sum_{\substack{0\le r\le p-1,\,0\le l\le 3\\ r+pl\equiv k(\text{mod }q)}} c_r m_{l,\beta}(\xi_r) z^l.
\end{align}
Then the system \eqref{system_of_eqns} becomes
$$\sum_{k=0}^{q-1} D_k(z\,;\xi) e^{-2\pi i sk/q}=0, \quad s=0,\dots,q-1.$$
Since the matrix $\big[e^{-2\pi i sk/q}\big]_{s,k=0}^{q-1}$ is invertible, it immediately follows that
\begin{align}\label{D_k=0}
    D_k(z\,;\xi) = 0
\end{align}
in \eqref{D_k equation} for all $k=0, 1, \dots, q-1$.

Now we divide the proof of Theorem \ref{p/qin(1/2,2/3)} into three cases, according to the structure of the associated $\Phi(z\,;\xi)$ matrix. Specifically, we consider rational densities in the intervals
$$\left(0,\frac{1}{2}\right),\left(\frac12,\frac35\right),
\left(\frac35,\frac23\right),\left(\frac23,\frac57\right),\quad\text{and}\quad \alpha\beta=\frac35, \frac{5}{7}$$
separately.

\begin{proof}[\textbf{Proof of Theorem \ref{p/qin(1/2,2/3)} for $\bm{0<\frac{p}{q}<\frac12}$}]
Let $\alpha\beta=\frac{p}{q}<\frac12$, where $\gcd(p,q)=1$. From \eqref{D_k=0}, we have
$$
D_k(z;\xi)=0,\qquad k=0,\ldots,q-1,
$$
where
$$
D_k(z;\xi)
=
\sum_{\substack{0\le r\le p-1,\;0\le l\le3\\
r+pl\equiv k\;(\mathrm{mod}\ q)}}
c_r m_{l,\beta}(\xi_r)z^l,
\qquad
\xi_r=\xi+\frac rp .
$$
By Lemma \ref{symmetry_singular_values}, it is enough to consider
$$
0\le \xi\le \frac{1}{2p}.
$$

Since $\frac{p}{q}<\frac12$, we have $q>2p$. let
$$
d:=q-2p\ge1.
$$
From the system $D_k(z;\xi)=0$, we select the $p$ equations indexed by
$$
k=p+j,\qquad j=0,\ldots,p-1.
$$

For a fixed $j$, consider the congruence
$$
r+pl\equiv p+j\pmod q,
\qquad
0\le r\le p-1,\quad 0\le l\le3.
$$
If $l=0$, then $r=p+j$, which is impossible since $r\le p-1$. If $l=1$, then
$$
r+p\equiv p+j\pmod q,
$$
and hence $r=j$. If $l=2$, then
$$
r+2p\equiv p+j\pmod q,
$$
so
$$
r\equiv j-p\equiv q-p+j\pmod q.
$$
Since $q>2p$, we have
$$
q-p+j>p-1,
$$
and therefore there is no admissible $r\in\{0,\ldots,p-1\}$. Finally, if $l=3$, then
$$
r+3p\equiv p+j\pmod q,
$$
and hence
$$
r\equiv j-2p\equiv q-2p+j=j+d\pmod q.
$$
Since $j+d<q$, this gives an admissible index precisely when
$j+d\le p-1.$

Consequently,
\[
D_{p+j}(z;\xi)
=
c_jm_{1,\beta}(\xi_j)z
+
\mathbf{1}_{\{j+d\le p-1\}}\,
c_{j+d}m_{3,\beta}(\xi_{j+d})z^3
=0,
\qquad j=0,\ldots,p-1.
\]
Moreover, since $z\in\mathbb T$ and
\[
0\le \xi_j
=
\xi+\frac{j}{p}
\le
\frac{1}{2p}+\frac{p-1}{p}
<1,
\]
we have
\[
m_{1,\beta}(\xi_j)<0,
\qquad j=0,\ldots,p-1.
\]

We now proceed by descending induction on $j$. For $j=p-1$, we have
$j+d>p-1$ since $d\ge1$, and hence the second term in
$D_{2p-1}(z;\xi)$ is absent. Therefore,
\[
c_{p-1}m_{1,\beta}(\xi_{p-1})z=0,
\]
which implies $c_{p-1}=0$. Suppose now that, for some $0\le j\le p-2$,
\[
c_{j+1}=c_{j+2}=\cdots=c_{p-1}=0.
\]
Since $d\ge1$, we have $j+d>j$. If $j+d\le p-1$, then
$c_{j+d}=0$ by the induction hypothesis; if $j+d>p-1$, the
corresponding term in $D_{p+j}(z;\xi)$ is absent. Thus, in either
case,
\[
D_{p+j}(z;\xi)
=
c_jm_{1,\beta}(\xi_j)z
=0.
\]
Since $m_{1,\beta}(\xi_j)\neq0$ and $z\in\mathbb T$, it follows that
$c_j=0$. Hence, by descending induction,
\[
c_0=c_1=\cdots=c_{p-1}=0.
\]
Therefore,
\[
\ker\Phi(z;\xi)=\{\bm 0\},
\qquad
(z,\xi)\in
\mathbb T\times
\left[0,\frac{1}{2p}\right].
\]
Hence, the Gabor system
$\mathcal{G}(g,\alpha,\beta)$ is a frame whenever
$$
0<\alpha\beta=\frac pq<\frac12.
$$
\end{proof}

\begin{pf}[\textbf{Proof of Theorem \ref{p/qin(1/2,2/3)} for} $\bm{\alpha\beta=\frac35, \frac{5}{7}}$]
   Here $p = 2N+1$ and $q = 2N+3$ for $N=1$ and $N=2$. Since $p = q-2$ and $\gcd(2, q) = 1$, the congruence $r + pl \equiv k \pmod q$ is equivalent to $r\equiv k+2l \pmod q$, and hence \eqref{D_k equation} becomes
\begin{align}
    D_k (z\,;\xi)=& \sum_{\substack{0 \le l \le 3 \\ r \equiv k+2l (\text{mod }q)}} c_r m_{l,\beta}(\xi_r) z^l=0, \nonumber
\end{align}
for all $k = 0, \dots, q-1$.

As the sequence of variables $c_r$ is only defined for $0 \le r \le p-1$, we set $c_{q-1}=0=c_{q-2}$. To systematically construct the subsystem, we index the rows by setting
$$k_m = 2m-2 \pmod q, \quad m = 0, \dots, p-1.$$
We select and order the subsystem equations $D_{k_m}(z\,;\xi) = 0$ according to this sequence. Let $A_N(\xi)$ denote the $p \times p$ coefficient matrix of this ordered subsystem, whose columns are ordered as 
$$(c_0, c_2, \ldots, c_{2N}, c_1, c_3, \ldots, c_{2N-1})^T.$$
Next, we introduce the substitution $c_{2j}=z^{-j}u_{2j}$ for $j=0,...,N$ and $c_{2j-1}=z^{-(N+1+j)}u_{2j-1}$ for $j=1,...,N$. We now distinguish the parity of $r$. If $r=2j$,
    then $2j \equiv 2m - 2 + 2l \pmod q$ reduces to $j = m - 1 + l$ and therefore,
    \begin{equation}c_{2j} z^l = z^{-j} u_{2j} z^l = z^{l-j} u_{2j} = z^{1-m} u_{2j}.\end{equation}
    If $r=2j-1$, then $2j-1 \equiv 2m - 2 + 2l \pmod q$ is equivalent to $2j-1=2m-2+2l-q$, yielding $j = m + l - N - 2$. Hence,
    \begin{equation}
    c_{2j-1} z^l = z^{-(N+1+j)} u_{2j-1} z^l = z^{l - N - 1 - j} u_{2j-1}=z^{1 - m} u_{2j-1}.
    \end{equation}
    Thus, every nonzero term in the $m$-th equation contains the common factor $z^{1-m}$. Since $|z|=1$, dividing by $z^{1-m}$ yields a system independent of $z$,
    \begin{align}\label{reduced system}
        A_N(\xi) \,\mathbf{u}=\mathbf{0},
    \end{align}
    where $\mathbf{u}=(u_0, u_2,\dots,u_{2N},u_1,u_3,\dots, u_{2N-1})^T$. \\
\noindent
\textbf{Case 1: When $N=1$.} Here $p/q=3/5$.
    \begin{align*}
        A_1(\xi)=\left[\begin{array}{cc|c}
m_{1,\beta}(\xi_0) & m_{2,\beta}(\xi_2) & 0\\
m_{0,\beta}(\xi_0) & m_{1,\beta}(\xi_2) & m_{3,\beta}(\xi_1)\\
\hline
0 & m_{0,\beta}(\xi_2) & m_{2,\beta}(\xi_1)
\end{array}\right].
    \end{align*}
    Here,
    \begin{align}
    \det A_1(\xi)
=
m_{1,\beta}(\xi_0)
m_{0,\beta}(\xi_2)
m_{0,\beta}(1-\xi_1)
\left[
1-F_\beta(\xi_2)F_\beta(1-\xi_1)
\right]-
m_{2,\beta}(\xi_2)
m_{0,\beta}(\xi_0)
m_{2,\beta}(\xi_1),
\end{align}
where 
$$F_\beta(x)=\frac{m_{1,\beta}(x)}{m_{0,\beta}(x)}$$
and $\xi_2=\frac23+\xi, 1-\xi_1=\frac23-\xi.$
Hence
$$F_\beta(\xi_2)F_\beta(1-\xi_1)=F_\beta\left(\frac23+\xi\right)F_\beta\left(\frac23-\xi\right).
$$
Since $F_\beta$ is strictly decreasing by Lemma \ref{Fx function} and $0\le \xi\le \frac{1}{6}$,
$$F_\beta\left(\frac23+\xi\right)
\ge F_\beta\left(\frac56\right)\qtq{and}F_\beta\left(\frac23-\xi\right)
\ge F_\beta\left(\frac23\right).$$
Therefore
$$F_\beta\left(\frac23+\xi\right)
F_\beta\left(\frac23-\xi\right)
\ge
F_\beta\left(\frac56\right)F_\beta\left(\frac23\right).$$
Using the uniform lower bounds in Lemma \ref{Fx function}, we have
$$F_\beta\left(\frac56\right)>\frac{21}{25},\quad F_\beta\left(\frac23\right)>\frac94.$$
Consequently, we obtain
\begin{align}\label{bound1}
F_\beta(\xi_2)F_\beta(1-\xi_1)>\frac{21}{25}\frac94>\frac{189}{100}>1.
\end{align}
Thus
$$1-F_\beta(\xi_2)F_\beta(1-\xi_1)<0.$$
Using the sign of $m_{l,\beta}(\xi)$ for $0< \xi\le \frac{1}{6}$, we obtain
$$\det A_1(\xi)>0.$$
At $\xi=0$, $m_{0,\beta}(\xi_0)=0$ and \eqref{bound1} gives
$$F_\beta(2/3)^2>1.$$
Then $\det A_1(0)>0$. Hence,
$$\det A_1(\xi)>0,
\quad
0\le\xi\le\frac16,\quad \beta>0.$$
\\
\noindent
    \textbf{Case 2: When $N=2$.} In this case, $p=5$ and $q=7$.
    \begin{align*}
        A_2(\xi)=
\left[\begin{array}{ccc|cc}
m_{1,\beta}(\xi_0) &
m_{2,\beta}(\xi_2) &
m_{3,\beta}(\xi_4) &
0 &
0
\\
m_{0,\beta}(\xi_0) &
m_{1,\beta}(\xi_2) &
m_{2,\beta}(\xi_4) &
0 &
0
\\
0 &
m_{0,\beta}(\xi_2) &
m_{1,\beta}(\xi_4) &
m_{3,\beta}(\xi_1) &
0
\\
\hline
0 &
0 &
m_{0,\beta}(\xi_4) &
m_{2,\beta}(\xi_1) &
m_{3,\beta}(\xi_3)
\\
0 &
0 &
0 &
m_{1,\beta}(\xi_1) &
m_{2,\beta}(\xi_3)
\end{array}\right].
    \end{align*}
    Here, $A_2(\xi)$ can be partitioned as a $2\times 2$ block matrix 
    \begin{align}\label{block_A2}
        A_2(\xi)=
\left[\begin{array}{c|c}
E_{2,\beta}(\xi) & W_{2,\beta}(\xi)\\[0.25em]
\hline
V_{2,\beta}(\xi) & O_{2,\beta}(\xi)
\end{array}\right],
    \end{align}
    where $E_{2,\beta}(\xi)$ is the upper-left $3\times3$ block and $O_{2,\beta}(\xi)$ is the lower-right $2\times2$ block. The off-diagonal blocks are the rank-1 matrices $W_{2,\beta}(\xi) = m_{3,\beta}(\xi_1) u_{3} v_1^T$ and $V_{2,\beta}(\xi) = m_{0,\beta}(\xi_{4}) v_1 u_{3}^T$, where $u_k$ and $v_k$ denote the $k$-th standard basis vectors in $\mathbb{C}^{3}$ and $\mathbb{C}^2$, respectively.

 To prove our result, it is enough to show that $\det A_2(\xi)\ne 0$ for $\xi\in \left[0,\frac{1}{10}\right]$. Applying the Schur complement gives 
\begin{align}
    \det A_2(\xi)
&=\det O_{2,\beta}(\xi)\,\det E_{2,\beta}(\xi)\left(1-\mathbf{Y}_\beta(\xi)\right),
\end{align}
where $\mathbf{Y}_\beta(\xi)=m_{3,\beta}(\xi_1) m_{0,\beta}(\xi_{4}) (E_{2,\beta}^{-1}(\xi))_{33} (O_{2,\beta}^{-1}(\xi))_{11}$.  Here 
    $$\det O_{2,\beta}(\xi)=m_{2,\beta}(\xi_1)m_{2,\beta}(\xi_3)
-m_{3,\beta}(\xi_3)m_{1,\beta}(\xi_1).$$
Since $m_2>0,~ m_3>0,~m_1<0,$ hence
$$\det O_{2,\beta}(\xi)>0\qtq{and} (O_{2,\beta}^{-1}(\xi))_{11}=
\frac{m_2(\xi_3)}{\det O_{2,\beta}(\xi)}>0.$$
Now compute $\det E_{2,\beta}(\xi)$. We have
$$\det E_{2,\beta}(0)=m_{1,\beta}(0)[m_{1,\beta}(\xi_2)m_{1,\beta}(\xi_4)-m_{0,\beta}(\xi_2)m_{2,\beta}(\xi_4)]<0$$
with 
$$\left(E_{2,\beta}(0)\right)^{-1}_{33}=\frac{m_{1,\beta}(\xi_2)}{m_{1,\beta}(\xi_2)m_{1,\beta}(\xi_4)-m_{0,\beta}(\xi_2)m_{2,\beta}(\xi_4)}.$$
For $\xi>0$,
$$\det E_{2,\beta}(\xi)=
m_{0,\beta}(\xi_0)m_{0,\beta}(\xi_2)m_{0,\beta}(\xi_4)\,H_\beta(\xi),$$
with
\begin{align}\label{H_beta function}
H_\beta(\xi)
=
F_\beta(\xi_0)F_\beta(\xi_2)F_\beta(\xi_4)
+
\Psi_\beta(1-\xi_2)F_\beta(\xi_4)+
\left(F_\beta(\xi_0)\Psi_\beta(1-\xi_4)-
\frac{F_\beta(\xi_4)}{\Psi_\beta(\xi_4)}\right).
\end{align}
Since 
$$F_\beta(\xi_0)>F_\beta(\xi_4),\quad \Psi_\beta(1-\xi_4)\ge 1,\qtq{and} \frac{1}{\Psi_\beta(\xi_4)}\le 1,$$
we have $H_\beta(\xi)>0$ in \eqref{H_beta function}. Therefore,
$$\det E_{2,\beta}(\xi)<0\qtq{and}(E_{2,\beta}^{-1}(\xi))_{3,3}<0.$$
To prove our result, it remains to prove 
$$\mathbf{Y}_\beta(\xi)\ne 1,\text{ for }\xi \in \left[0,\frac{1}{10}\right].$$
Since $m_{3,\beta}(\xi_1) > 0$, $m_{0,\beta}(\xi_{4}) < 0$, $\mathbf{Y_\beta(\xi)}$ is strictly positive. 
Consequently, 
$$Y_\beta(\xi)=
\frac{1}{
F_\beta(\xi_4)
F_\beta(1-\xi_1)
(1+R_E)(1+R_O)
},$$
where $F_\beta(\xi)$, and $\Psi_\beta(\xi)$ are in Lemmas \ref{Fx function} and \ref{Psi_function} respectively, 
$$R_O=
\frac{
\Psi_\beta(\xi_1)
}{
F_\beta(1-\xi_1)
F_\beta(1-\xi_3)
}>0,$$
and
$$R_E=
\frac{
F_\beta(\xi_0)\Psi_\beta(1-\xi_4)-
F_\beta(\xi_4)/\Psi_\beta(\xi_4)
}{
F_\beta(\xi_4)
\left[
F_\beta(\xi_0)F_\beta(\xi_2)
+
\Psi_\beta(1-\xi_2)
\right]
}.$$
Since 
$$F_\beta(\xi_0)>F_\beta(\xi_4),\quad \Psi_\beta(1-\xi_4)\ge1,
\qtq{and}
\frac1{\Psi_\beta(\xi_4)}\le1,
$$
we have $R_E>0$. Since $R_O>0$,
\begin{align*}
    0<\mathbf{Y}_\beta(\xi )&\le \frac{1}{
F_\beta(\xi_4)
F_\beta(1-\xi_1)
(1+R_E)}.
\end{align*}
It is enough to show that
$$K_\beta(\xi)=F_\beta(\xi_4)
F_\beta(1-\xi_1)
(1+R_E)>1,~0\le \xi\le\frac{1}{10}.$$
Moreover,
$$K_\beta(\xi) = F_\beta\left(\frac45-\xi\right) F_\beta\left(\frac45+\xi\right) \left[ 1+ \frac{F_\beta(\frac15-\xi)-1/F_\beta(\xi)}{\Psi_\beta(\frac45+\xi) \left[ F_\beta(\frac25+\xi) + \Psi_\beta(\frac35-\xi)/F_\beta(\xi) \right]} \right]$$
From Lemma \ref{K_beta_greater_than_1}, we have $K_\beta(\xi)>1$, for $0\le \xi\le \frac{1}{10}$. Since $\mathbf{Y}_\beta(\xi )$ is continuous at $\xi=0$, taking limit  $\xi\rightarrow 0^+$ we obtain 
$0<Y_\beta(0)\le \frac{1}{K_\beta(0)}$.
From Lemma \ref{K_beta_greater_than_1}, we obtain $1-Y_\beta(0)>0.$ Thus,
$$\det A_2(\xi)\ne 0,\qtq{for}\xi\in \left[0,\frac{1}{10}\right].$$
Hence, the system \eqref{reduced system} admits only the trivial solution $\mathbf{u}=\mathbf{0}$, which implies that $\mathbf{c}=\mathbf{0}$. Therefore,  $$\ker\Phi_{\alpha}(z; \xi) = \{\mathbf{0}\}, \text{ for all }(z, \xi) \in \mathbb{T} \times \left[0, \frac{1}{2p}\right],$$
completing the proof.
\end{pf}

\begin{pf}[\textbf{Proof of Theorem \ref{p/qin(1/2,2/3)} for} $\bm{p/q\in (1/2,3/5)}$] From \eqref{D_k=0}, we have
\begin{align}\label{Thm interval eqn}
     D_k(z\,;\xi) = \sum_{\substack{0\le r\le p-1,\,0\le l\le 3\\ r+pl\equiv k(\text{mod }q)}} c_r m_{l,\beta}(\xi_r) z^l=0,
\end{align}
for the rational density $\alpha\beta=p/q$, $\frac{1}{2}< \frac{p}{q}< \frac{3}{5}$, $\xi_r=\xi+\frac{r}{p}$, and $\xi\in \left[0,\frac{1}{2p}\right]$. 

By analyzing the congruence relation $r+pl\equiv k \pmod q$ for $0\le r\le p-1$ under the assumption $1/2<p/q<3/5$, the active row indices for each shift $l\in\{0,1,2,3\}$ are determined as follows
$$\begin{aligned}  l=0:&\quad k\in[0,p-1], \\  l=1:&\quad k\in[p,q-1]\cup[0,2p-q-1], \\  l=2:&\quad k\in[2p-q, 3p-q-1], \\  l=3:&\quad k\in[3p-q, q-1]\cup[0,4p-2q-1].  \end{aligned}$$
Intersecting these active domains partitions the row index set into three classes of two-variable equations (Classes A, B, and C) together with the remaining three-variable rows, as illustrated in Figure~\ref{fig:partition}.

\begin{figure}[H]
    \centering
\includegraphics[width=\textwidth]{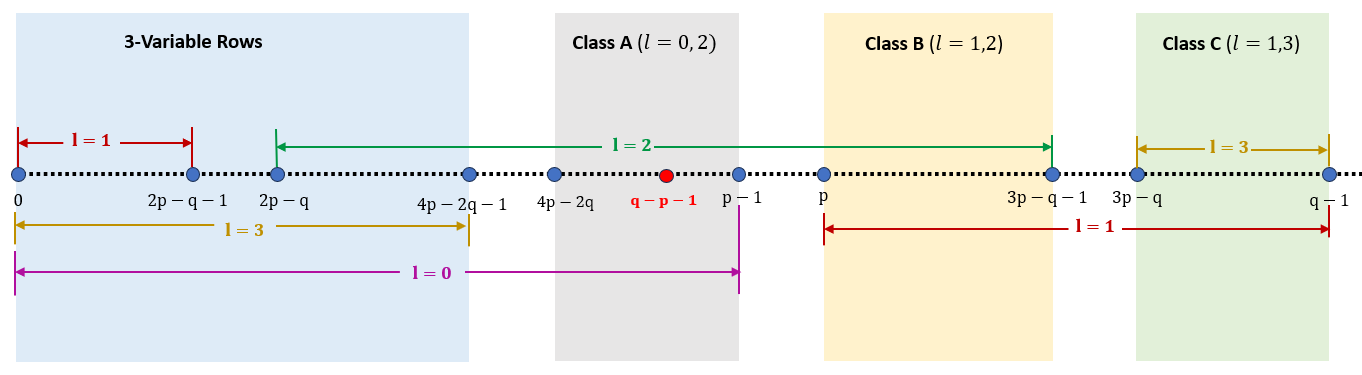}
    \caption{ Partitioning of the equation index space $k \in [0, q-1]$ by the polynomial shift parameter $l \in \{0, 1, 2, 3\}$ for rational densities $1/2 < p/q < 3/5$. The intersections of these active domains partition the system into 3-variable interior rows and precise 2-variable equation classes (Classes A, B, and C).}
    \label{fig:partition}
\end{figure}
\noindent
\textbf{Class A}: Only $l=0, 2$, are active, $k \in [4p-2q, p-1]$ giving $(p-1) - (4p-2q) + 1 = 2q - 3p$ equations.\\
\noindent
\textbf{Class B}: Only $l=1, 2$ are active, $k \in [p, 3p-q-1]$ giving $(3p-q-1) - p + 1 = 2p - q$ equations.\\
\noindent
\textbf{Class C}: Only $l=1, 3$ are active, $k \in [3p-q, q-1]$ giving $(q-1) - (3p-q) + 1 = 2q - 3p$ equations.

Hence, the total number of two-variable equations is
$$(2q-3p) + (2p-q) + (2q-3p) = 3q - 4p.$$
For $p/q<3/5$, we have $3q - 5p \ge 1$ which implies that $3q - 4p \ge p + 1.$ 

Now define a graph $G$ whose vertices are $V = \{c_0, c_1, \ldots, c_{p-1}\}$, where
edge $E_r$ is defined as
$$E_r=\{c_r, c_{(r+q-p)\text{ mod }p}\},~0\le r\le p-1.$$ 
Consequently, Class B gives $E_r$ for $0\le r\le 2p-q-1$, Class C gives $E_r$ for $2p-q\le r\le q-p-1$ and Class A gives $E_r$ for $4p-2q\le r\le p-1$. Therefore, every edge $E_r$ forms a two-variable equation with non-zero coefficients. Since $p/q<3/5$ implies $4p-2q\le q-p-1$, the above three classes together contain $E_r$ for $0\le r\le p-1$. Since $\gcd (p,q-p)=1$, graph $G$ is connected. 

Next, consider the equations corresponding to $k = q-1$ and $k = q-p-1$. These give
\begin{flalign*}
   &&D_{q-1}(z; \xi) &= c_{q-p-1} m_{1,\beta}(\xi_{q-p-1})z + c_{2q-3p-1} m_{3,\beta}(\xi_{2q-3p-1})z^3 = 0&\\
   \text{and}&&D_{q-p-1}(z; \xi) &= c_{q-p-1} m_{0,\beta}(\xi_{q-p-1}) + c_{2q-3p-1} m_{2,\beta}(\xi_{2q-3p-1})z^2 = 0.
\end{flalign*}
This can be written as 
$$\begin{pmatrix} 
 m_{1,\beta}(\xi_{q-p-1})z & m_{3,\beta}(\xi_{2q-3p-1})z^3\\
    m_{0,\beta}(\xi_{q-p-1}) & m_{2,\beta}(\xi_{2q-3p-1})z^2
    \end{pmatrix} \begin{pmatrix} c_{q-p-1} \\ c_{2q-3p-1} \end{pmatrix} = \begin{pmatrix} 0 \\ 0 \end{pmatrix}.$$
The determinant of the coefficient matrix is
$$\Delta=m_{0,\beta}(\xi_{q-p-1})m_{3,\beta}(\xi_{2q-3p-1})z^3[F_\beta(\xi_{q-p-1})F_\beta(1-\xi_{2q-3p-1})-1],$$
where $F_\beta$ is defined in Lemma \ref{Fbeta_beta_monotonicity}. From Lemma \ref{Fbeta_beta_monotonicity}, we obtain
\begin{align*}
    F_\beta(\xi_{q-p-1})F_\beta(1-\xi_{2q-3p-1})>F_0(\xi_{q-p-1})F_0(1-\xi_{2q-3p-1})>1.
\end{align*}
Let $x = \xi_{q-p-1}$ and $y = 1-\xi_{2q-3p-1}$. Moreover, 
$$2(1-x) - y = \frac{1}{p} - \xi.$$
Since $\xi \le \frac{1}{2p}$, we have $\frac{1}{p} - \xi > 0$. Therefore, $y < 2(1-x)$
implies $x < 1 - \frac{y}{2}$. Since $1/2<p/q < 3/5$,
$2p - q \ge 1$ and $3q - 5p \ge 1$ enforce $p \ge 4$. Consequently, $\frac{1}{3p} \le \frac{1}{12}$, guaranteeing
$$0 < y \le \frac{2}{3} + \frac{1}{12} = \frac{3}{4}.$$
Since $F_0(t)$ is strictly decreasing on $(0,1)$, we have
$$F_0(x)F_0(y) > F_0\left(1 - \frac{y}{2}\right)F_0(y).$$
From Lemma \ref{Fbeta_beta_monotonicity}, recall that $F_0(t) = \frac{(1-t)(3t+1)}{t^2}$. Therefore,
$$F_0\left(1 - \frac{y}{2}\right)F_0(y) - 1 = \frac{8y^3 - 26y^2 + 9y + 8}{y(2-y)^2}.$$
For $0 < y \le 3/4$, $y^3 > 0$ and $y^2 \le \frac{3}{4}y$ yield
$$8y^3 - 26y^2 + 9y + 8 > -26y^2 + 9y + 8 \ge -\frac{39}{2}y + 9y + 8 = 8 - \frac{21}{2}y.$$
Since $y \le 3/4$, we obtain
$$8 - \frac{21}{2}y \ge 8 - \frac{63}{8} = \frac{1}{8} > 0.$$
Thus, $F_0(1-y/2)F_0(y) > 1$, which ensures $F_0(x)F_0(y) > 1$.
Hence 
$$F_{\beta}(\xi_{q-p-1})F_{\beta}(1-\xi_{2q-3p-1}) > 1$$
and $\Delta\ne 0$, forcing 
$$c_{q-p-1} = 0 \quad \text{and} \quad c_{2q-3p-1} = 0.$$
Since $G$ is connected, repeated application of the two-variable relations forces
$$c_0 = c_1 = \cdots = c_{p-1} = 0,$$
which completes the proof.
\end{pf}

\begin{pf}[\textbf{Proof of Theorem \ref{p/qin(1/2,2/3)} for} $\bm{p/q\in (3/5,2/3)}$]
From \eqref{D_k=0}, we have
\begin{align}\label{Subsystem_in_(3/5,2/3)}
     D_k(z\,;\xi) = \sum_{\substack{0\le r\le p-1,\,0\le l\le 3\\ r+pl\equiv k(\text{mod }q)}} c_r m_{l,\beta}(\xi_r) z^l=0,
\end{align}
for the rational density $\alpha\beta=p/q\in (3/5,2/3)$, $\xi_r=\xi+\frac{r}{p}$, and $\xi\in \left[0,\frac{1}{2p}\right]$. By analyzing the congruence relation $r+pl\equiv k \pmod q$ for $0\le r\le p-1$ under the assumption $3/5 < p/q < 2/3$, the active row indices for each shift $l\in\{0,1,2,3\}$ are determined in Figure \ref{Thm 1.3 for first half fig}.
\begin{figure}[H]
        \centering
\includegraphics[width=\textwidth]{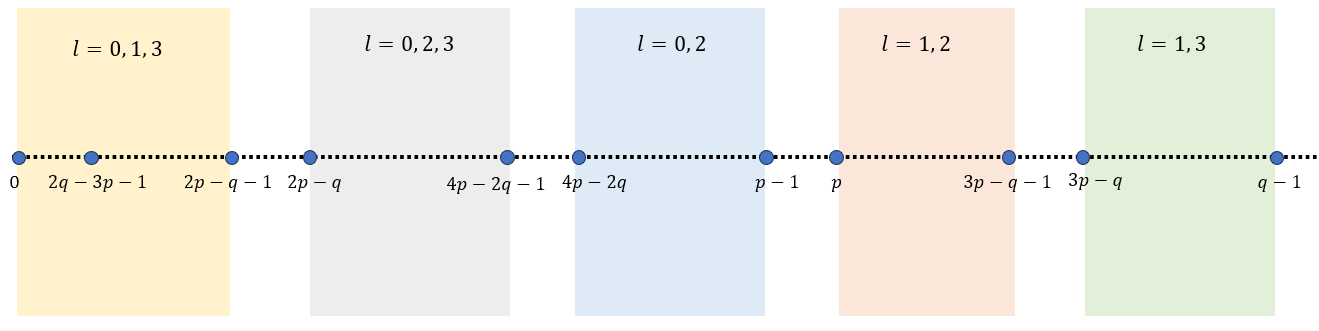}
        \caption{Partitioning of the equation index space $k \in [0, q-1]$ by the polynomial shift parameter $l \in \{0, 1, 2, 3\}$ for rational densities $3/5 < p/q < 2/3$.}
        \label{Thm 1.3 for first half fig}
\end{figure}
For rational densities in this interval where $2q-3p=1$, the boundary evaluation at $\xi=0$ requires a carefully selected sub-matrix. Excluding this specific boundary case, we can extract the following closed $3 \times 3$ subsystem directly from \eqref{Subsystem_in_(3/5,2/3)} 
    \begin{flalign*}
       && D_{q-1} &= c_{q-p-1}m_{1,\beta}(\xi_{q-p-1})z + c_{2q-3p-1}m_{3,\beta}(\xi_{2q-3p-1})z^3 = 0&\\
        &&D_{q-p-1}&= c_{q-p-1}m_{0,\beta}(\xi_{q-p-1}) + c_{2q-3p-1}m_{2,\beta}(\xi_{2q-3p-1})z^2 + c_{3q-4p-1}m_{3,\beta}(\xi_{3q-4p-1})z^3 = 0\\
       \text{and} &&D_{2q-2p-1}&= c_{2q-3p-1}m_{1,\beta}(\xi_{2q-3p-1})z + c_{3q-4p-1}m_{2,\beta}(\xi_{3q-4p-1})z^2 = 0,
    \end{flalign*}
    for $3/5<p/q<2/3$. By introducing the index notation $A = q-p-1$, $B = 2q-3p-1$, and $C = 3q-4p-1$, this system can be compactly expressed in matrix form as
      $$\begin{bmatrix}
      m_{1,\beta}(\xi_A)z & m_{3,\beta}(\xi_B)z^3 & 0 \\ m_{0,\beta}(\xi_A) & m_{2,\beta}(\xi_B)z^2 & m_{3,\beta}(\xi_C)z^3 \\ 0 & m_{1,\beta}(\xi_B)z & m_{2,\beta}(\xi_C)z^2 \end{bmatrix} \begin{bmatrix} c_{q-p-1} \\ c_{2q-3p-1} \\ c_{3q-4p-1} \end{bmatrix} = \begin{bmatrix} 0 \\ 0 \\ 0 \end{bmatrix}.$$
      Here, $0\le B<A<C<p-1$.
    The determinant of the above $3\times 3$ coefficient matrix $M$ is 
    $$\det M= z^5 \left[ m_{1,\beta}(\xi_A)m_{2,\beta}(\xi_B)m_{2,\beta}(\xi_C) - m_{1,\beta}(\xi_A)m_{1,\beta}(\xi_B)m_{3,\beta}(\xi_C) - m_{0,\beta}(\xi_A)m_{3,\beta}(\xi_B)m_{2,\beta}(\xi_C) \right].$$
    If $\det M=0$, then
\begin{align}\label{thm 1.4 second interval}
m_{1,\beta}(\xi_A)m_{1,\beta}(\xi_B)m_{3,\beta}(\xi_C) + m_{0,\beta}(\xi_A)m_{3,\beta}(\xi_B)m_{2,\beta}(\xi_C)&=m_{1,\beta}(\xi_A)m_{2,\beta}(\xi_B)m_{2,\beta}(\xi_C) \nonumber\\
\frac{m_{1,\beta}(1-\xi_B)}{m_{0,\beta}(1-\xi_B)} + \frac{m_{1,\beta}(\xi_B)}{m_{0,\beta}(1-\xi_B)} \frac{m_{0,\beta}(1-\xi_C)}{m_{1,\beta}(1-\xi_C)}&= \frac{m_{0,\beta}(\xi_A)}{m_{1,\beta}(\xi_A)} \nonumber\\
      F_\beta(1-\xi_B) + \frac{\Psi_\beta(\xi_B)}{F_\beta(1-\xi_C)} &= \frac{1}{F_\beta(\xi_A)},
    \end{align}
    where $F_\beta$ and $\Psi_\beta$ are defined in Lemmas \eqref{Fbeta_beta_monotonicity} and \eqref{Psi_function} respectively. From Lemma \ref{Last_lemma_for_p/q}, we can show that \eqref{thm 1.4 second interval} never holds. Therefore,
$\det M$ is
nonzero. 
    
    Now consider $\xi=0$ and $2q-3p=1$; then we need to select the equations $D_{q-2}$, $D_{q-p-2}$, and $D_{2q-2p-2}$, then the closed $3 \times 3$ subsystem is 
    $$\left[\begin{matrix}m_{1,\beta}(\xi _{D})z&m_{2,\beta}(\xi _{I})z^{2}&0\\ m_{0,\beta}(\xi _{D})&m_{1,\beta}(\xi _{I})z&m_{3,\beta}(\xi _{J})z^{3}\\ 0&m_{0,\beta}(\xi _{I})&m_{2,\beta}(\xi _{J})z^{2}\end{matrix}\right]\left[\begin{matrix}c_{q-p-2}\\ c_{2q-2p-2}\\ c_{3q-4p-2}\end{matrix}\right]=\left[\begin{matrix}0\\ 0\\ 0\end{matrix}\right],$$
    where
    $D=q-p-2$, $I=2q-2p-2$, and $J=3q-4p-2$. Then 
    $$\det =z^{4}\left[m_{1,\beta}(\xi _{D})m_{1,\beta}(\xi _{I})m_{2,\beta}(\xi _{J})-m_{1,\beta}(\xi _{D})m_{0,\beta}(\xi _{I})m_{3,\beta}(\xi _{J})-m_{0,\beta}(\xi _{D})m_{2,\beta}(\xi _{I})m_{2,\beta}(\xi _{J})\right].$$
    Suppose
    \begin{align}\label{Thm second xi=0}
        \frac{m_{1,\beta}(\xi _{I})}{m_{0,\beta}(\xi _{I})}-\frac{m_{3,\beta}(\xi _{J})}{m_{2,\beta}(\xi _{J})}&=\frac{m_{0,\beta}(\xi _{D})m_{2,\beta}(\xi _{I})}{m_{1,\beta}(\xi _{D})m_{0,\beta}(\xi _{I})}\nonumber\\
        F_\beta(\xi _{I})+\frac{ \Psi_\beta (1-\xi _{I})}{F_\beta(\xi _{D})}&=\frac{1}{F_\beta(1-\xi _{J})}.
    \end{align}
 For $3/5<p/q<2/3$, $\xi=0$ and $2q-3p=1$, we have 
 $\xi_D\in [0.2,0.5)$, $\xi_I\in [0.8,1)$, and $\xi_J\in [0.4,0.5)$. Let $1-\xi_I=t\in (0,\frac{1}{2}]$, then we have $\xi_D=\frac{1}{2}-\frac{3t}{2}$ and $\xi_J=\frac{1}{2}-\frac{t}{2}$. Now from Lemma \ref{E_beta}, we conclude that \eqref{Thm second xi=0} does not hold.

Hence,  the determinant of the coefficient subsystem being non-zero implies
$$c_{q-p-1}=c_{2q-3p-1}=c_{3q-4p-1}=0$$
whenever $(\xi,2q-3p)\neq (0,1)$, whereas in the exceptional case $\xi=0$ and $2q-3p=1$, it yields
$$c_{q-p-2}=c_{2q-2p-2}=c_{3q-4p-2}=0.$$

Now we can show that all remaining coefficients must vanish by exploiting the recursive structure of the equations $D_k(z;\xi)=0$. As before, we extend the coefficient sequence by setting
$$c_r=0,~ p\le r\le q-1.$$
Let 
$$
\delta= 
\begin{cases} 1, & (\xi,2q-3p)\neq(0,1),\\ 2, & \xi=0,\ 2q-3p=1. 
\end{cases}$$
Since $\gcd(p,q)=1$, for every fixed integer $\delta$, the map
$$m\longmapsto mp-\delta \pmod q$$
is a permutation of $\mathbb Z_q$. Define
$v_m=c_{(mp-\delta)\bmod q}.$
\\
\noindent
\textbf{Case 1}: Take $(\xi,2q-3p)\neq(0,1)$. In this case, define
$$v_m=c_{(mp-1)\bmod q},~ m\in\mathbb Z_q.$$
Then
\begin{align*} v_{-4}&=c_{3q-4p-1}=0,\quad v_{-3}=c_{2q-3p-1}=0,\\ v_{-2}&=c_{2q-2p-1}=0,\quad v_{-1}=c_{q-p-1}=0,\quad v_0=c_{q-1}=0. 
\end{align*}
\textbf{Case 2}: Take $\xi=0$ and $2q-3p=1$. In this case
$$v_m=c_{(mp-2)\bmod q}.$$
Then using $2q-3p=1$ we obtain
\begin{align*} 
v_{-4}&=c_{3q-4p-2}=0,\quad v_{-3}=c_{q-1}=0,\\ 
v_{-2}&=c_{2q-2p-2}=0,\quad v_{-1}=c_{q-p-2}=0,\quad v_0=c_{q-2}=0. \end{align*}

For $m\in\mathbb Z$, set
$$r_j^{(m)}=(m-j)p-\delta\pmod q, ~ j=0,1,2,3.$$

Taking $k=(mp-\delta)\pmod q$ in $D_k(z;\xi)=0$, we obtain the recurrence
\begin{align}\label{recurrence_eq}
v_m m_{0,\beta}(\xi_{r_0^{(m)}}) +v_{m-1}m_{1,\beta}(\xi_{r_1^{(m)}})z +v_{m-2}m_{2,\beta}(\xi_{r_2^{(m)}})z^2 +v_{m-3}m_{3,\beta}(\xi_{r_3^{(m)}})z^3 =0.
\end{align}

Assume inductively that
$v_{m-1}=v_{m-2}=v_{m-3}=0.$
If either $\xi>0$ or $r_0^{(m)}\neq0$, then
$\xi_{r_0^{(m)}}>0,$
and hence
$m_{0,\beta}(\xi_{r_0^{(m)}})\neq0.$
Therefore, \eqref{recurrence_eq} immediately gives $v_m=0.$

It remains only for $\xi=0,~ r_0^{(m)}=0.$
In this case $mp-\delta\equiv0\pmod q,$ and hence
$(m+1)p-\delta\equiv p\pmod q.$
Consequently, $v_{m+1}=c_p=0$ by our extension of the coefficient sequence. Applying \eqref{recurrence_eq} with $m$ replaced by $m+1$, the coefficient of $v_m$ is $m_{1,\beta}(0)\neq0.$
Since $v_{m+1}=v_{m-1}=v_{m-2}=0$, the recurrence again yields
$v_m=0.$

Starting from $v_{-2}=v_{-1}=v_0=0,$
and applying this argument successively, we obtain
$v_1=v_2=\cdots=v_{q-1}=0.$ Thus
$$v_m=0,~ m\in\mathbb Z_q.$$
Since the map
$m\longmapsto mp-\delta\pmod q$
is a permutation of $\mathbb Z_q$, it follows that
$$c_0=c_1=\cdots=c_{p-1}=0.$$
Therefore,
$$\ker\Phi(z;\xi)=\{\mathbf 0\},$$
which completes the proof.
\end{pf}

\begin{pf}[\textbf{Proof of Theorem \ref{p/qin(1/2,2/3)} for} $\bm{p/q\in (2/3,5/7)}$]
Now, we extract the following closed subsystem of four equations from \eqref{D_k=0} that
 \begin{flalign*}
    &&D_{q-1}&= c_{q-p-1}m_{1,\beta}(\xi _{q-p-1})z+c_{2q-2p-1}m_{2,\beta}(\xi _{2q-2p-1})z^{2}=0&\\
 &&D_{q-p-1}&= c_{q-p-1}m_{0,\beta}(\xi _{q-p-1})+c_{2q-2p-1}m_{1,\beta}(\xi _{2q-2p-1})z+c_{3q-4p-1}m_{3,\beta}(\xi _{3q-4p-1})z^{3}=0&\\
 &&D_{3q-3p-1}&= c_{3q-4p-1}m_{1,\beta}(\xi _{3q-4p-1})z+c_{4q-5p-1}m_{2,\beta}(\xi _{4q-5p-1})z^{2}=0&\\
 \text{and}&&D_{2q-2p-1}&= c_{2q-2p-1}m_{0,\beta}(\xi _{2q-2p-1})+c_{3q-4p-1}m_{2,\beta}(\xi _{3q-4p-1})z^{2}+c_{4q-5p-1}m_{3,\beta}(\xi _{4q-5p-1})z^{3}=0,
 \end{flalign*}
for $2/3<p/q<5/7$. This subsystem can be expressed in matrix form as
$$\left[\begin{matrix}m_{1,\beta}(\xi _{q-p-1})z&m_{2,\beta}(\xi _{2q-2p-1})z^{2}&0&0\\ m_{0,\beta}(\xi _{q-p-1})&m_{1,\beta}(\xi _{2q-2p-1})z&m_{3,\beta}(\xi _{3q-4p-1})z^{3}&0\\ 0&0&m_{1,\beta}(\xi _{3q-4p-1})z&m_{2,\beta}(\xi _{4q-5p-1})z^{2}\\ 0&m_{0,\beta}(\xi _{2q-2p-1})&m_{2,\beta}(\xi _{3q-4p-1})z^{2}&m_{3,\beta}(\xi _{4q-5p-1})z^{3}\end{matrix}\right]\left[\begin{matrix}c_{q-p-1}\\ c_{2q-2p-1}\\ c_{3q-4p-1}\\ c_{4q-5p-1}\end{matrix}\right]=\left[\begin{matrix}0\\ 0\\ 0\\ 0\end{matrix}\right].$$
Suppose that the coefficient matrix is singular. Expanding its
determinant and using the identities in \eqref{reflection-symmetry}, after dividing
by the nonzero common factors, we obtain
\begin{align}\label{Thm rational in (2/3,5/7)}
F_\beta(1-x_3)F_\beta(x_2)
+
\Psi_\beta(x_3)
\frac{F_\beta(x_2)}{F_\beta(1-x_4)}
+
\Psi_\beta(1-x_2)
\frac{F_\beta(1-x_3)}{F_\beta(x_1)}
+
\frac{
\Psi_\beta(x_3)\Psi_\beta(1-x_2)}
{F_\beta(x_1)F_\beta(1-x_4)}
=1,
\end{align}
where $x_1=\xi_{q-p-1}$, $x_2=\xi_{2q-2p-1}$, $x_3=\xi_{3q-4p-1}$, $x_4=\xi_{4q-5p-1},
$ and $F_\beta(x)$ is defined in Lemma \eqref{Fbeta_beta_monotonicity}. We obtain from Lemma \ref{T1_and_T3} that the Left-Hand Side of \eqref{Thm rational in (2/3,5/7)} is strictly greater than $1$; this contradiction proves that the coefficient matrix is non-singular and strictly invertible, forcing 
$$c_{q-p-1} = c_{2q-2p-1} = c_{3q-4p-1} = c_{4q-5p-1} = 0.$$
To show that the remaining coefficients vanish, extend the sequence by
setting $c_r=0,\qquad p\le r\le q-1,$
and define
$$
v_m=c_{(mp-1)\bmod q}.
$$
This gives
$$
v_{-5}=v_{-4}=v_{-3}=v_{-2}=v_{-1}=v_0=0.
$$
Since $p\le 3q-3p-1\le q-1$ for $2/3<p/q<5/7$, we have $v_{-3}=c_{3q-3p-1}=0$. Therefore, applying the recurrence
\eqref{recurrence_eq} with $\delta=1$, together with the same argument for the exceptional case
$m_{0,\beta}(0)=0$, yields
$$
v_m=0,\qquad m\in\mathbb Z_q.
$$
Since $m\mapsto mp-1\pmod q$ is a permutation of $\mathbb Z_q$, we
conclude that
$$
c_0=c_1=\cdots=c_{p-1}=0.
$$
Hence, $\ker\Phi(z;\xi)=\{\bm 0\}$, completing the proof.
\end{pf}

Theorem \ref{p/qin(1/2,2/3)} leads us to the following Conjecture.
\begin{conj}
The frame set of $g$ is
    $$\mathcal{F}(g)=\left\{(\alpha,\beta)\in \R^2_{+}:~\alpha\beta < 1\text{ and }\alpha\beta \neq \frac{N}{N+1},\text{ for every }N \ge 1\right\}.$$
\end{conj}
The remaining rational densities are strongly supported by numerical computations of the associated Zibulski–Zeevi matrices, but a complete analytic proof of the nonvanishing of the full matrix is complicated.

\appendix
\renewcommand{\thelemma}{\thesection\arabic{lemma}}
\section{Appendix}
\begin{lem}\label{lem:R_k_increasing}
For $k \ge 1$ and $x > 1$, define
$$R_k(x) = \frac{(2x+1)^{2k} + (2x-1)^{2k} - (x+2)^{2k} - (x-2)^{2k}}{4^k - 1} - (x-1)^{2k}.$$
Then $R_{k+1}(x) > R_k(x)$.
\end{lem}

\begin{pf}
Set $P_k(x) = R_k(x) + (x-1)^{2k}$. Using binomial expansion, we obtain
$$P_k(x) = 2\sum_{j=0}^k d_{k,j} x^{2j}, ~ \text{where} \quad d_{k,j} = \binom{2k}{2j} \frac{4^j - 4^{k-j}}{4^k - 1},$$
and $P_k(1) = 0$. Define $S_k(x) = P_{k+1}(x) - x^2 P_k(x) = \sum_{j=0}^{k+1} c_j x^{2j}$. We have $c_0 = -2$, $c_{k+1} = 0$, and for $1 \le j \le k$
$$c_j = 2\binom{2k}{2j-2} \frac{4^{j-1}}{4^k - 1} T_j,$$
where $T_j = \Lambda_j(1 - t_j) + 4t_j - 1$, with $t_j = 4^{k+1-2j}$ and
$$\Lambda_j = 4\frac{4^k - 1}{4^{k+1} - 1} \frac{(2k+2)(2k+1)}{2j(2j-1)}.$$
The sequence $\Lambda_j$ is strictly decreasing with respect to $j$, and
$$\Lambda_j\ge \Lambda_k=4\frac{4^k - 1}{4^{k+1} - 1} \frac{(2k+2)(2k+1)}{2k(2k-1)}>1,~k\ge 1.$$
We distinguish two cases based on $t_j$.\\
\noindent
\textbf{Case 1}: If $t_j \le 1$, then $\Lambda_j > 1$, and hence
$$T_j = (\Lambda_j - 1)(1 - t_j) + 3t_j > 0.$$
\\
\noindent
\textbf{Case 2}:
If $t_j > 1$, then we can rewrite
\begin{align}\label{sign}
    T_j = (t_j - 1) \left( 4 + \frac{3}{t_j - 1} - \Lambda_j \right).
\end{align}
As $j$ increases, $t_j$ strictly decreases. Consequently, the term $4 + \frac{3}{t_j - 1}$ is strictly increasing, while $\Lambda_j$ is strictly decreasing. Hence, \eqref{sign} shows that the signs of $T_j$, and therefore of $c_j$, can change at most once, and only from negative to positive. We already know $c_0 < 0$, $c_k > 0$, and the rest of the coefficient sequence possesses exactly one sign change. 

By Descartes' rule of signs, $S_k(x)$ has exactly one positive zero. Since $S_k(1) = P_{k+1}(1) - P_k(1) = 0$, that unique positive zero is exactly $x=1$. Thus, for all $x > 1$, we have $S_k(x) > 0$.

Finally, 
$$R_{k+1}(x) - x^2 R_k(x) = S_k(x) + (2x-1)(x-1)^{2k}> 0.$$
Hence, $R_{k+1}(x) > x^2 R_k(x)$. Moreover $R_1(x) = (x-1)(x+3) > 0$ for $x>1$. By induction, this guarantees $R_k(x) > 0$. Therefore, for $x^2 > 1$, we conclude
$$R_{k+1}(x) > x^2 R_k(x) > R_k(x).$$
\end{pf}

\begin{lem}\label{Fbeta_beta_monotonicity}
For $\beta>0$, define
\begin{align}
    F_\beta(\xi)=\frac{m_{1,\beta}(\xi)}{m_{0,\beta}(\xi)},
    \quad \xi\in(0,1).\nonumber
\end{align}
For any fixed $\xi\in(0,1)$, the function $\beta\longmapsto F_\beta(\xi)$ is strictly increasing on $(0,\infty)$. Consequently,$$F_\beta(\xi) > \lim_{\beta\to 0^+} F_\beta(\xi) = \frac{(1-\xi)(3\xi+1)}{\xi^2}.$$
\end{lem}
\begin{pf}
Let $a=\pi\beta$. Recall that
\begin{align*}
m_{0,\beta}(\xi)
&=
-\frac{2}{3}
\sinh(\pi\beta\xi)\sinh(3\pi\beta\xi),
\end{align*}
 and
\begin{align*}
m_{1,\beta}(\xi)
=&m_{0,\beta}(\xi+1)-2\big(
\cosh(2\pi\beta)+\cosh(4\pi\beta)
\big)m_{0,\beta}(\xi).
\end{align*}
Using the product-to-sum formulas for hyperbolic functions and simplifying, we can express $F_\beta(\xi)$ can be written as
$$F_\beta(\xi) = \frac{N(a)}{D(a)},$$
where
\begin{align*}
N(a) =& \frac{1}{2}\big[ \cosh(2a(2+\xi)) + \cosh(2a(2-\xi)) + \cosh(2a(1-\xi)) \\
&\hspace{2cm} - \cosh(2a(2\xi+1)) - \cosh(2a(1-2\xi)) - \cosh(4a(1-\xi)) \big],
\end{align*}
and
\begin{align*}
D(a) =& \frac{1}{2}\big[ \cosh(4a\xi) - \cosh(2a\xi) \big].
\end{align*}
Expanding the hyperbolic cosines into their Maclaurin
series, we obtain
$$
N(a)
=
\sum_{k=1}^{\infty}
\frac{(2a)^{2k}}{(2k)!}\,C_k(\xi),
\quad
D(a)
=
\sum_{k=1}^{\infty}
\frac{(2a)^{2k}}{(2k)!}\,d_k(\xi),
$$
where
\begin{align*}
C_k(\xi)
=&
(2+\xi)^{2k}
+(2-\xi)^{2k}
+(1-\xi)^{2k}-(2\xi+1)^{2k}
-(1-2\xi)^{2k}
-(2-2\xi)^{2k},
\end{align*}
and
$$
d_k(\xi)
=
(2\xi)^{2k}-\xi^{2k}
=
\xi^{2k}(4^k-1)>0.
$$

We use the following elementary power-series ratio principle,
if
$$
A(t)=\sum_{k\ge1}a_k t^k,
\quad
B(t)=\sum_{k\ge1}b_k t^k,
\quad b_k>0,
$$
and the sequence $a_k/b_k$ is strictly increasing, then
$A(t)/B(t)$ is strictly increasing for $t>0$
throughout the common interval of convergence.

In the present case, the common positive factor
$\frac{2^{2k}}{(2k)!}$
cancels in the coefficient ratios. Thus it suffices to show that $\frac{C_k(\xi)}{d_k(\xi)}$ is strictly increasing with respect to $k$.

Set $x=\frac1{\xi}>1$. Factoring $\xi^{2k}$ from $C_k(\xi)$ gives
$$
\frac{C_k(\xi)}{d_k(\xi)}
=
R_k(x),
$$
where
$$
R_k(x)
=
\frac{
(2x+1)^{2k}
+(2x-1)^{2k}
-(x+2)^{2k}
-(x-2)^{2k}}
{4^k-1}
-(x-1)^{2k}.
$$
By Lemma~\ref{lem:R_k_increasing},
$$
R_{k+1}(x)>R_k(x),
\quad x>1.
$$
Hence
$$
\frac{C_{k+1}(\xi)}{d_{k+1}(\xi)}
>
\frac{C_k(\xi)}{d_k(\xi)},
\quad k\ge1.
$$

Applying the power-series ratio principle with
$t=a^2$ therefore shows that
$$
a\longmapsto \frac{N(a)}{D(a)}
$$
is strictly increasing for $a>0$. Since $a=\pi\beta$,
we conclude that
$$
\beta\longmapsto F_\beta(\xi)
$$
is strictly increasing on $(0,\infty)$. Consequently,
$$
F_\beta(\xi)
>
\lim_{\beta\to0^+}F_\beta(\xi)=\frac{(1-\xi)(3\xi+1)}{\xi^2},
\quad
\beta>0,\quad0<\xi<1.
$$
\end{pf}

\begin{lem}\label{Fx function}
For $\beta>0$, define
\begin{align}
F_\beta(\xi)=\frac{m_{1,\beta}(\xi)}{m_{0,\beta}(\xi)},
 \quad \xi\in(0,1).\nonumber
\end{align}
Then $F_\beta$ is strictly decreasing and positive on $(0,1)$.
Moreover,
\begin{align}\label{Fbeta_bounds}
 F_\beta\left(\frac{4}{5}\right)>1,\quad   F_\beta\left(\frac{2}{3}\right)>\frac94,
    \quad
    F_\beta\left(\frac{5}{6}\right)>\frac{21}{25},
    \qtq{for} \beta>0.
\end{align}
\end{lem}

\begin{pf}
From the definition of $m_{0,\beta}(\xi)$ and $m_{1,\beta}(\xi)$, we have
\begin{align}\label{Fbeta_explicit}
F_\beta(\xi)
=&
\frac{m_{0,\beta}(\xi+1)}
     {m_{0,\beta}(\xi)}
-
2\big(
\cosh(2\pi\beta)+\cosh(4\pi\beta)
\big)
\nonumber\\
=&
\frac{
\sinh\big(\pi\beta(\xi+1)\big)
\sinh\big(3\pi\beta(\xi+1)\big)
}{
\sinh(\pi\beta\xi)
\sinh(3\pi\beta\xi)
}-2\big(
\cosh(2\pi\beta)+\cosh(4\pi\beta)
\big).
\end{align}

To prove monotonicity, observe that
\begin{align*}
\frac{m_{0,\beta}'(\xi)}
     {m_{0,\beta}(\xi)}
=
\pi\beta
\left[
\coth(\pi\beta\xi)
+
3\coth(3\pi\beta\xi)
\right].
\end{align*}
Since $\coth x$ is strictly decreasing on $(0,\infty)$,
the function
\begin{align*}
\xi\longmapsto
\frac{m_{0,\beta}'(\xi)}
     {m_{0,\beta}(\xi)}
\end{align*}
is strictly decreasing. Therefore,
\begin{align*}
\left(
\frac{m_{0,\beta}(\xi+1)}
     {m_{0,\beta}(\xi)}
\right)'
&=
\frac{m_{0,\beta}(\xi+1)}
     {m_{0,\beta}(\xi)}
\left[
\frac{m_{0,\beta}'(\xi+1)}
     {m_{0,\beta}(\xi+1)}
-
\frac{m_{0,\beta}'(\xi)}
     {m_{0,\beta}(\xi)}
\right]
<0.
\end{align*}
Hence $F_\beta$ is strictly decreasing on $(0,1)$. From \eqref{Fbeta_explicit},
\begin{align*}
F_\beta(1)
=&
\frac{
\sinh(2\pi\beta)\sinh(6\pi\beta)}
{\sinh(\pi\beta)\sinh(3\pi\beta)}-
2\big(
\cosh(2\pi\beta)+\cosh(4\pi\beta)
\big)=0.
\end{align*}
Since $F_\beta$ is strictly decreasing on $(0,1)$,
it follows that
$$
F_\beta(\xi)>F_\beta(1)=0,
\quad 0<\xi<1.
$$
Thus $F_\beta$ is strictly decreasing and positive on $(0,1)$.

As $\beta\to0^+$, we obtain
\begin{align}\label{Fbeta_limit}
\lim_{\beta\to0^+}F_\beta(\xi)
&=
\frac{(\xi+1)^2}{\xi^2}-4
=
\frac{(1-\xi)(3\xi+1)}{\xi^2}.
\end{align}
In particular,
\begin{align*}
\lim_{\beta\to0^+}
F_\beta\left(\frac23\right)
=\frac94,\quad
\lim_{\beta\to0^+}
F_\beta\left(\frac45\right)
=\frac{17}{16},\qtq{and}
\lim_{\beta\to0^+}
F_\beta\left(\frac56\right)
=\frac{21}{25}.
\end{align*}
From Lemma \ref{Fbeta_beta_monotonicity}, we conclude lower bounds in \eqref{Fbeta_bounds}.
\end{pf}

\begin{lem}\label{lem:reciprocal_gap_monotone}
For $0<a<b<1$, the function
$$
\beta\longmapsto
\frac{1}{F_\beta(a)}-\frac{1}{F_\beta(b)}
$$
is strictly increasing on $(0,\infty)$.
\end{lem}

\begin{pf}
Let
$$
A=\pi\beta,
\quad
U_A(x)=\frac{1}{F_\beta(x)},
\quad 0<x<1.
$$
By the Fundamental Theorem of Calculus, we have
$$\frac{\partial}{\partial \beta} \left( \frac{1}{F_\beta(x_1)} - \frac{1}{F_\beta(x_2)} \right) = -\pi \int_{x_1}^{x_2} \partial_x\partial_A U_A(x) \, dx.$$
It is enough to prove that
\begin{equation}\label{eq:mixed_U_negative}
\partial_x\partial_A U_A(x)<0,
\quad A>0,\quad 0<x<1.
\end{equation}
Let
$$
Y=\tanh(Ax),
\quad
Z=\tanh(A(1-x)).
$$
Then
$$
0<Y,Z<1,
\quad
Ax=\operatorname{arctanh}Y,
\quad
A(1-x)=\operatorname{arctanh}Z.
$$
Then
\begin{equation}\label{eq:U_YZ}
U_A(x)
=
\frac{
Y^2(1-Y^2)(Y^2+3)(1-Z^2)^2
}{
Z\,H(Y,Z)
},
\end{equation}
where
\begin{align*}
H(Y,Z)
=&
3Y^6Z^3+9Y^6Z+20Y^5Z^2+12Y^5
+11Y^4Z^3+33Y^4Z\\
&+56Y^3Z^2+8Y^3
+17Y^2Z^3+51Y^2Z
+20YZ^2+12Y+Z^3+3Z.
\end{align*}
In particular,
$$
H(Y,Z)>0,\quad 0<Y,Z<1.
$$

Let
$$
\alpha=\operatorname{arctanh}Y=Ax,
\quad
\gamma=\operatorname{arctanh}Z=A(1-x).
$$
Exact differentiation of \eqref{eq:U_YZ} gives
\begin{equation}\label{eq:mixed_N}
\partial_x\partial_AU_A(x)
=
\frac{N(Y,Z)}{Z^3H(Y,Z)^3},
\end{equation}
where
$$
N(Y,Z)=C_0(Y,Z)+\alpha C_1(Y,Z)+\gamma C_2(Y,Z).
$$
Define
$$
h(t)=\operatorname{arctanh}t-t.
$$
Since $\alpha=Y+h(Y)$ and $\gamma=Z+h(Z)$,
we may rewrite
\begin{equation}\label{eq:N_B}
N(Y,Z)
=
B(Y,Z)+C_1(Y,Z)h(Y)+C_2(Y,Z)h(Z),
\end{equation}
where
$B=C_0+YC_1+ZC_2.$ The two coefficients needed below admit the exact factorizations
\begin{align}
C_2
=&
-2Y(Y-1)(Y+1)(Z-1)^2(Z+1)^2P_2(Y,Z),
\label{eq:C2_factor}
\\
B
=&
-2YZ(Y-1)(Y+1)(Y+Z)(Z-1)^2(Z+1)^2P_B(Y,Z).
\label{eq:B_factor}
\end{align}
The explicit polynomial certificates given in \ref{polynomial-certificates} show that
\begin{equation}\label{eq:P2_PB_sign}
P_2(Y,Z)<0,
\quad
P_B(Y,Z)<0,
\quad 0<Y,Z<1.
\end{equation}

Since $Y-1<0$ and all the remaining factors in
\eqref{eq:C2_factor}--\eqref{eq:B_factor} are positive,
\eqref{eq:P2_PB_sign} implies
\begin{equation}\label{eq:B_C2_negative}
B(Y,Z)<0,
\quad
C_2(Y,Z)<0.
\end{equation}

We now use
$$
h(t)
=
\sum_{k=1}^{\infty}\frac{t^{2k+1}}{2k+1},
\quad 0<t<1.
$$
Hence
\begin{equation}\label{eq:h_bounds}
\frac{t^3}{3}
<
h(t)
<
\frac{t^3}{3(1-t^2)}.
\end{equation}

We distinguish two cases.
\\
\noindent
\textbf{Case 1}: If $C_1(Y,Z)\le0$, then $h(Y),h(Z)>0$, and
\eqref{eq:N_B} together with \eqref{eq:B_C2_negative} gives
$$
N(Y,Z)
<
B(Y,Z)<0.
$$
\textbf{Case 2}: If $C_1(Y,Z)>0$. Using \eqref{eq:h_bounds}, we obtain
\begin{align}
N(Y,Z)
&<
B+C_1\frac{Y^3}{3(1-Y^2)}+C_2\frac{Z^3}{3}=R(Y,Z).
\label{eq:R_upper}
\end{align}
Therefore,
\begin{equation}\label{eq:R_factor}
R(Y,Z)=-\frac23\,YZ(1-Z^2)^2P_R(Y,Z).
\end{equation}
The remaining polynomial satisfies
\begin{equation}\label{eq:PR_positive}
P_R(Y,Z)>0,
\quad 0<Y,Z<1.
\end{equation}
For completeness, this is proved in \ref{polynomial-certificates_P_R}. From \eqref{eq:R_factor}, we have
$R(Y,Z)<0.$
Consequently, in the case $C_1>0$, \eqref{eq:R_upper} gives
$N(Y,Z)<R(Y,Z)<0.$

Combining the two cases,
$$
N(Y,Z)<0,
\quad 0<Y,Z<1.
$$
Finally, since $Z^3H(Y,Z)^3>0$, equation
\eqref{eq:mixed_N} yields
$$
\partial_x\partial_AU_A(x)<0,
\quad A>0,\quad 0<x<1.
$$
Therefore
$$
\frac{\partial}{\partial\beta}
\left(
\frac1{F_\beta(a)}-\frac1{F_\beta(b)}
\right)>0,
$$
and hence
$$
\beta\longmapsto
\frac1{F_\beta(a)}-\frac1{F_\beta(b)}
$$
is strictly increasing on $(0,\infty)$.
\end{pf}

\subsection{Sign of \texorpdfstring{$P_2$}{P2} and \texorpdfstring{$P_B$}{PB}}\label{polynomial-certificates}

Collecting $-P_2(Y,Z)$ in powers of $Z$, write
$$
-P_2(Y,Z)=\sum_{j=0}^{8}a_j(Y)Z^j.
$$
An exact symbolic expansion, reproduced by the \href{https://github.com/Riya74012/Gabor-rational-windows-symbolic-verification}
{MATLAB} code, shows that all $a_0,\ldots,a_7$ are strictly positive. The only negative
monomial occurs in
$$a_8(Y)=Y\left(
37Y^{12}+157Y^{10}+273Y^8+1603Y^6
+1471Y^4+495Y^2+75-15Y^{14}
\right).
$$
Since $0<Y<1$,
$$
37Y^{12}-15Y^{14}
=
Y^{12}(37-15Y^2)>0.
$$
Consequently $a_j(Y)>0$ for every $j=0,\ldots,8$, and hence
$$
-P_2(Y,Z)>0,
\quad 0<Y,Z<1.
$$
Thus $P_2(Y,Z)<0.$

Similarly, 
$$
-P_B(Y,Z)=\sum_{j=0}^{7}b_j(Y)Z^j.
$$
An exact symbolic expansion, reproduced by the \href{https://github.com/Riya74012/Gabor-rational-windows-symbolic-verification}
{MATLAB} code shows that
$b_0,b_1,b_4,b_5,b_6$ have nonnegative coefficients and are
nonzero. The only negative monomials occur in $b_2,b_3,b_7$.
More precisely,
\begin{align}\label{terms}
    37Y^{13}-15Y^{15},
\quad
3921Y^{13}-75Y^{15},
\quad
8356Y^{12}-78Y^{14}.
\end{align}
terms in \eqref{terms} are positive for $Y\in (0,1)$. Consequently, $b_j(Y)>0$
for every coefficient appearing in the expansion of $-P_B$.
Therefore,
$$
-P_B(Y,Z)>0,
\quad 0<Y,Z<1,
$$
and hence
$$
P_B(Y,Z)<0.
$$

\subsection{Positivity of \texorpdfstring{$P_R$}{PR}}\label{polynomial-certificates_P_R}

Put $t=Y^2\in(0,1)$. Collecting in powers of $Z$, write
$$
P_R(Y,Z)
=
\sum_{j=0}^{10}Y^{\varepsilon_j}R_j(t)Z^j,
\quad
\varepsilon_j\in\{0,1\}.
$$

For $j=0,\ldots,9$, the coefficient polynomials admit 
$$R_j(t) = (1-t)A_j(t)+c_jt^{n_j},$$
where $A_j(t)$ is a polynomial with nonnegative coefficients with at least one strictly positive coefficient, $c_j\ge0$. Therefore,
$$R_j(t)>0, \quad 0<t<1,\quad j=0,\ldots,9.$$
A \href{https://github.com/Riya74012/Gabor-rational-windows-symbolic-verification}
{MATLAB} is available for verification.
For the last coefficient, 
$$R_{10}(t) = (1-t) \Big( 75+495t+1471t^2+1603t^3 +273t^4+157t^5+37t^6-15t^7 \Big).$$
Since $0<t<1$,
$$37t^6-15t^7 = t^6(37-15t)>0.$$
Hence $R_{10}(t)>0$, for $0<t<1$.
Therefore,
$$R_j(t)>0, \quad j=0,\ldots,10.$$
Hence, 
$$P_R(Y,Z)>0,~0<Y,Z<1.$$

\begin{lem}\label{Psi_function}
Define
$$\Psi_\beta(\xi)=\frac{m_{1,\beta}(\xi)}{m_{0,\beta}(1-\xi)}.$$
Then $\Psi_\beta(\xi)$ is strictly increasing on $[0,1)$. Consequently, $\Psi_\beta(\xi)\ge 1$.
\end{lem}
\begin{pf}
Let $U = e^{2\pi\beta}$ and $X=e^{2\pi\beta \xi} \in [1, U)$. Since $\frac{dX}{d\xi}=2\pi \beta X>0$, it is sufficient to prove that $\frac{d\Psi_\beta}{dX}>0$. From
$$m_{0,\beta}(\xi)=\frac16\left(X+\frac1X-X^2-\frac1{X^2}\right),$$
we obtain
$$m_{0,\beta}(\xi+1)
=\frac16\left(UX+\frac1{UX}-U^2X^2-\frac1{U^2X^2}\right).$$
Therefore, 

    \begin{align*}
        \Psi_\beta(X)&= \frac{m_{0,\beta}(XU) - \left(U+U^2+\frac1U+\frac1{U^2}\right) \cdot m_{0,\beta}(X)}{m_{0,\beta}\left(\frac{U}{X}\right)}\\
        &= \frac{X^3(U^3+U + 1) + X^2(U^2  - 1) + X(U^3 - U) - (U^3 + U^2 + 1)}{U^3 - X^3}.
    \end{align*}
    Differentiating $\Psi_\beta(X)$ with respect to $X$, we have
    $$\frac{d\Psi_\beta}{dX} = \frac{\left(U^{2} - 1\right) \left(X^{4} + 2UX^{3} + \left(3U^{4} + 6U^{2} + 3\right) X^{2} + 2U^{3} X + U^{4}\right)}{\left(X^{3} - U^{3}\right)^{2}}.$$
    Since $\beta>0$, $U > 1$, and $X \ge 1$, the leading coefficient $(U^2- 1)$ is strictly positive, and every term within the expanded polynomial in the numerator is strictly positive. Moreover, $X < U$, the denominator term $U^3 - X^3 \neq 0$. Therefore, $\frac{d\Psi_\beta}{dX} > 0$. Consequently, we have $\Psi_\beta(0)=1$. Since $\Psi_\beta$ is strictly increasing, $\Psi_\beta(\xi)\ge 1$, for $\xi\in [0,1)$.
\end{pf}

\begin{lem}\label{K_beta_greater_than_1}
Let $\beta > 0$ and $0 \le \xi \le 1/10$. For $0 < \xi \le 1/10$, define
$$K_\beta(\xi) = F_\beta\left(\frac45-\xi\right) F_\beta\left(\frac45+\xi\right) \left[ 1+ \frac{F_\beta(\frac15-\xi)-1/F_\beta(\xi)}{\Psi_\beta(\frac45+\xi) \left[ F_\beta(\frac25+\xi) + \Psi_\beta(\frac35-\xi)/F_\beta(\xi) \right]} \right].$$
Then
$$K_\beta(\xi) > 1, \quad 0 \le \xi \le \frac{1}{10}.$$
\end{lem}

\begin{pf}
    First, consider $0 < \xi \le 1/10$. Let
    $$a = \frac15-\xi, \quad b = \frac45-\xi, \quad c = \frac25+\xi, \quad d = \frac35-\xi, \qtq{and} e = \frac45+\xi.$$
    Therefore,
    $0 < \xi \le a < d < b < e < 1.$
    We have
    $$\Psi_\beta(x)\Psi_\beta(1-x) = F_\beta(x)F_\beta(1-x).$$
    Let 
    $$C_\beta(\xi) = F_\beta(e) \frac{F_\beta(a)-1/F_\beta(\xi)}{\Psi_\beta(e) \left[ F_\beta(c)+\Psi_\beta(d)/F_\beta(\xi) \right]}.$$
    Since $a = 1-e$, we have $\frac{F_\beta(e)}{\Psi_\beta(e)} = \frac{\Psi_\beta(a)}{F_\beta(a)}$. Substituting this yields
    $$C_\beta(\xi) = \frac{\Psi_\beta(a) \left( 1-\frac{1}{F_\beta(\xi)F_\beta(a)} \right)}{F_\beta(c)+\frac{\Psi_\beta(d)}{F_\beta(\xi)}}.$$
Since $d = 1-c$ and $\Psi_\beta(c) \ge 1$, identity $\Psi_\beta(d)\Psi_\beta(c) = F_\beta(d)F_\beta(c)$ gives $\Psi_\beta(d) \le F_\beta(d)F_\beta(c)$. For $\xi \le 1/10$ and $d = 3/5-\xi \ge 1/2$, it follows that $\xi < d$. Since $F_\beta$ is strictly decreasing, $F_\beta(d) < F_\beta(\xi)$, which guarantees $\Psi_\beta(d) < F_\beta(\xi)F_\beta(c)$. Consequently
\begin{align}\label{bound_denom}
    F_\beta(c)+ \frac{\Psi_\beta(d)}{F_\beta(\xi)} < 2F_\beta(c).
\end{align}
For $0 < \xi \le 1/10$, monotonicity ensures $F_\beta(\xi) > F_0(\xi) \ge F_0(1/10) = 117$. Additionally, $1/10 \le a < 1/5$, ensuring $F_\beta(a) > F_0(a) \ge F_0(1/5) = 32$. Hence, $F_\beta(\xi)F_\beta(a) > 117 \cdot 32 > 5$. Therefore
\begin{align}\label{bound_neum}
    \frac12 \left( 1-\frac{1}{F_\beta(\xi)F_\beta(a)} \right) > \frac12\left(1-\frac15\right) = \frac25. 
\end{align}

Since $\Psi_\beta(a) \ge 1$, combining equations \eqref{bound_denom} and \eqref{bound_neum} establishes that $$C_\beta(\xi) > \frac{2}{5F_\beta(c)}.$$
Substituting this into the original $K_\beta(\xi)$ gives
\begin{align}\label{bound_k}
    K_\beta(\xi) > F_\beta(b) \left[ F_\beta(e) + \frac{2}{5F_\beta(c)} \right]. 
\end{align}
It remains to show that the right-hand side is strictly greater than $1$. Define 
$$J_\beta(\xi) = F_\beta(e) + \frac{2}{5F_\beta(c)} - \frac{1}{F_\beta(b)}.$$
Then \eqref{bound_k} simplifies to
\begin{align}\label{final_bound_k}
    K_\beta(\xi) > 1 + F_\beta(b)J_\beta(\xi).
\end{align}
Since $F_\beta(b) > 0$, it is sufficient to prove $J_\beta(\xi) > 0$. We can rewrite
$$J_\beta(\xi) = F_\beta(e) + \frac{2}{5}\left(\frac{1}{F_\beta(c)} - \frac{1}{F_\beta(b)}\right)-\frac{3}{5}\frac{1}{F_\beta(b)}.$$ 
Since $c<b$, from Lemmas \ref{lem:reciprocal_gap_monotone} and \ref{Fbeta_beta_monotonicity}, we obtain
$$J_\beta(\xi) > J_0(\xi).$$
Using $F_0(x) = \frac{(1-x)(3x+1)}{x^2}$ yields
\begin{align}\label{J0_eqn}
    J_0(\xi) = \frac{-P(\xi)}{5(3-5\xi)(5\xi+1)(5\xi+4)^2(17-15\xi)(15\xi+11)},
\end{align}
where 
$$P(\xi) = 1968750\xi^6 + 387500\xi^5 - 3355000\xi^4 - 493500\xi^3 + 1007825\xi^2 - 23630\xi - 7621.$$
The denominator in \eqref{J0_eqn} is positive for $0 \le \xi \le 1/10$. Thus, we must show $P(\xi) < 0$. 

For $\xi \in [0, 1/10]$, $\xi^6 \le \xi^3/1000$ and $\xi^5 \le \xi^3/100$. We have a strictly bound $P(\xi) \le R(\xi)$, where
$$R(\xi) = 1007825\xi^2 - \frac{1950625}{4}\xi^3 - 23630\xi - 7621.$$
Since $R''(\xi) = 2015650 - \frac{5851875}{2}\xi\ge \frac{6892225}{4} > 0$ for $\xi \le 1/10$, $R(\xi)$ is strictly convex on $[0, 1/10]$, and its maximum is achieved at the boundary points. Since
$$R(0) = -7621 < 0, \quad R\left(\frac{1}{10}\right) = -\frac{12589}{32} < 0,$$
we have $R(\xi) < 0$, which implies that  $P(\xi) < 0$. Therefore, it confirms $J_0(\xi) > 0$ in \eqref{J0_eqn} and hence $J_\beta(\xi) > J_0(\xi) > 0$ implies that 
$$K_\beta(\xi) > 1+F_\beta(b)J_\beta(\xi) > 1, \quad 0 < \xi \le \frac{1}{10}.$$

Finally, at $\xi = 0$. Since $F_\beta(\xi) \to +\infty$ as $\xi \to 0^+$, we have $1/F_\beta(\xi) \to 0$. Therefore,
$$K_\beta(0) = F_\beta\left(\frac45\right)^2 \left[ 1+ \frac{F_\beta(\frac15)}{\Psi_\beta(\frac45)F_\beta(\frac25)} \right].$$
Since the fractional term in the bracket is strictly positive, the bracket term is strictly greater than $1$. Furthermore, $F_\beta(4/5) > 1$. Hence, $K_\beta(0) > 1$. Hence,
$$K_\beta(\xi) > 1 \quad \text{for every } \beta > 0, \quad 0 \le \xi \le \frac{1}{10}.$$
\end{pf}

\begin{lem}\label{E_beta}
Let $\beta>0$, and define
\begin{align}
E_\beta(t)
=&
F_\beta(1-t)
+
\frac{\Psi_\beta(t)}
{F_\beta\left(\frac12-\frac{3t}{2}\right)}
-
\frac{1}
{F_\beta\left(\frac12+\frac{t}{2}\right)},
\quad 0< t\le \frac15.\nonumber
\end{align}
Then
\begin{align}
E_\beta(t)>0, \quad 0< t\le \frac15.  \nonumber
\end{align}
\end{lem}

\begin{pf}
    Since $\Psi_\beta(t)\ge 1$, we have
    \begin{align}\label{Ebeta_lower1}
        E_\beta(t)\ge F_\beta(1-t)
+
\frac{1}
{F_\beta\left(\frac12-\frac{3t}{2}\right)}
-
\frac{1}
{F_\beta\left(\frac12+\frac{t}{2}\right)}.
    \end{align}
Let $x_t=\frac12-\frac{3t}{2}$, and $y_t=\frac12+\frac{t}{2}.$
Since $0<t\le 1/5$, we have
$$
0<x_t<y_t<1.
$$
By the monotonicity of $F_\beta(x)$ with respect to $\beta$,
$$
F_\beta(1-t)>F_0(1-t).
$$
Since $x_t<y_t$, Lemma \ref{lem:reciprocal_gap_monotone}
implies that
$$\frac1{F_\beta(x_t)}-\frac1{F_\beta(y_t)}
>\frac1{F_0(x_t)}-\frac1{F_0(y_t)}.
$$
Combining this with \eqref{Ebeta_lower1}, we obtain
\begin{align}
E_\beta(t)
>
F_0(1-t)
+
\frac1{F_0(x_t)}
-
\frac1{F_0(y_t)}.
\label{eq:Ebeta_lower2}
\end{align}

From Lemma \ref{Fx function}, recall that
$F_0(x)=\frac{(1-x)(3x+1)}{x^2}.$
Consequently,
$$
F_0(1-t)
=
\frac{t(4-3t)}{(1-t)^2},\quad
\frac1{F_0(x_t)}
=
\frac{(1-3t)^2}
{(1+3t)(5-9t)},
$$
and
$$
\frac1{F_0(y_t)}
=
\frac{(1+t)^2}
{(1-t)(5+3t)}.
$$
Therefore,
\begin{align}
F_0(1-t)
+
\frac1{F_0(x_t)}
-
\frac1{F_0(y_t)}&=
\frac{
t\left(
243t^4-21t^3-619t^2+217t+52
\right)
}{
(1-t)^2(1+3t)(5+3t)(5-9t)
}.
\label{eq:rational_lower_bound}
\end{align}

The denominator in \eqref{eq:rational_lower_bound} is strictly
positive for $0<t\le 1/5$.  Furthermore,
$$
t^2\le \frac{t}{5},
\quad
t^3\le\frac{t}{25},
$$
and hence
\begin{align*}
243t^4-21t^3-619t^2+217t+52
&\ge 52+\frac{2309}{25}t>0.
\end{align*}
Thus the right-hand side of \eqref{eq:rational_lower_bound}
is strictly positive. From \eqref{eq:Ebeta_lower2}, it follows that
$$
E_\beta(t)>0,
\quad
0<t\le\frac15.
$$
This completes the proof.
\end{pf}

\begin{lem}\label{T1_and_T3}
Let
$
\frac23<\frac pq<\frac57$, $0\le \xi\le \frac1{2p},$
and 
$$
x_1=\xi_{q-p-1},\quad
x_2=\xi_{2q-2p-1},\quad
x_3=\xi_{3q-4p-1}.
$$
Then, for every $\beta>0$,
$$
F_\beta(1-x_3)\left[F_\beta(x_2)+
\frac{\Psi_\beta(1-x_2)}{F_\beta(x_1)}\right]
>1.
$$
\end{lem}

\begin{pf}
Let $d=1-x_2$ and $y=1-x_3$. From the definitions of $x_1,x_2,x_3$, we have $y=x_1+2d$. In particular, $0<x_1<y<1$. Using $\Psi_\beta(d)\ge1$, we obtain
\begin{align}
F_\beta(y)F_\beta(x_2)
+
\Psi_\beta(d)\frac{F_\beta(y)}{F_\beta(x_1)}&\ge
F_\beta(y)
\left(
F_\beta(x_2)+\frac1{F_\beta(x_1)}
\right)
\nonumber\\
&=
1+F_\beta(y)
\left[
F_\beta(x_2)
+\frac1{F_\beta(x_1)}
-\frac1{F_\beta(y)}
\right].
\label{eq:T1T3_reduction}
\end{align}
Thus it is enough to prove
\begin{equation}\label{eq:Gbeta_positive}
G_\beta=
F_\beta(x_2)
+\frac1{F_\beta(x_1)}
-\frac1{F_\beta(y)}
>0.
\end{equation}

Since $F_\beta(x)>F_0(x)$ for every fixed $x\in(0,1)$ and
$\beta>0$, we have
$$
F_\beta(x_2)>F_0(x_2).
$$
Moreover, since $x_1<y$, Lemma~\ref{lem:reciprocal_gap_monotone}
gives
$$
\frac1{F_\beta(x_1)}
-\frac1{F_\beta(y)}
>
\frac1{F_0(x_1)}
-\frac1{F_0(y)}.
$$
Consequently,
\begin{equation}\label{eq:Gbeta_G0}
G_\beta>G_0,
\quad
G_0=
F_0(x_2)
+\frac1{F_0(x_1)}
-\frac1{F_0(y)}.
\end{equation}

It remains to show that $G_0>0$. Put
$k=3p-2q$, $l=5q-7p.$ Solving these two relations gives $p=5k+2l$.
Set $\tau=\frac1p$ and $c=1-p\xi$.
Since $0\le\xi\le1/(2p)$ and $l\ge1$,
$$\frac12\le c\le1,\quad 0<\tau\le\frac1{5k+2}.$$

A direct calculation gives
$$
d=(k+c)\tau,
\quad
x_2=1-(k+c)\tau,
$$
and
$$
x_1=\frac{1-(k+2c)\tau}{2},
\quad
y=\frac{1+(3k+2c)\tau}{2}.
$$

Using $F_0(x)=\frac{(1-x)(3x+1)}{x^2}$,
we obtain
\begin{equation}\label{eq:G0_factor}
G_0
=
-\frac{\tau(k+c)P_{k,c}(\tau)}
{T_{k,c}(\tau)},
\end{equation}
where
\begin{align*}
T_{k,c}(\tau)
=&
[1-(k+c)\tau]^2
[1+(k+2c)\tau]
[1-(3k+2c)\tau]
[5-(3k+6c)\tau]
[5+(9k+6c)\tau].
\end{align*}
All factors in $T_{k,c}(\tau)$ are strictly positive for
$k\ge1$, $1/2\le c\le1$, and
$\tau\le(5k+2)^{-1}$. Hence
$T_{k,c}(\tau)>0.$

Furthermore,
\begin{align}
P_{k,c}(\tau)
=&
A_5\tau^5-A_4\tau^4-A_3\tau^3+A_2\tau^2
+(123k-21c)\tau-52,
\label{eq:Pkc}
\end{align}
where
\begin{align*}
A_5=&
432c^5+2160c^4k+4104c^3k^2
+3672c^2k^3+1539ck^4+243k^5,
\\
A_4=&
640c^4+2704c^3k+3968c^2k^2
+2348ck^3+480k^4,\\
A_3=&
280c^3+584c^2k+438ck^2+182k^3,
\\
A_2=&
528c^2+868ck+460k^2.
\end{align*}

We claim that $P_{k,c}(\tau) < 0$. We evaluate this in two cases based on $k$.
\\
\noindent
\textbf{Case 1}: $k\ge2$. Since $1/2\le c\le1$, discarding the
negative terms $-A_4\tau^4-A_3\tau^3$ and using
$\tau\le(5k+2)^{-1}$ gives
\begin{align*}
P_{k,c}(\tau)
&<
\frac{A_5|_{c=1}}{(5k+2)^5}
+
\frac{A_2|_{c=1}}{(5k+2)^2}
+
\frac{123k}{5k+2}
-52
\\
&=
-\frac{
27882k^5+22961k^4-41272k^3
-54744k^2-21952k-2992
}{(5k+2)^5}.
\end{align*}
Putting $k=s+2$, $s\ge0$, the numerator becomes
$$
27882s^5+301781s^4+1257696s^3
+2479248s^2+2229120s+663552,
$$
which is strictly positive. Hence
$$
P_{k,c}(\tau)<0,
\quad k\ge2.
$$
\noindent
\textbf{Case 2}: $k=1$. In this case $\tau\le1/7$.
Since $1/2\le c\le1$, we have
$$
A_5\le12150,\quad
A_4\ge3024,\quad
A_3\ge582,\quad
A_2\le1856,\qtq{and}
123-21c\le\frac{225}{2}.
$$
Therefore
$$
P_{1,c}(\tau)\le H(\tau),
$$
where
$$
H(\tau)
=
12150\tau^5-3024\tau^4-582\tau^3
+1856\tau^2+\frac{225}{2}\tau-52.
$$
Moreover,
$$
H'(\tau)
=
60750\tau^4-12096\tau^3-1746\tau^2
+3712\tau+\frac{225}{2}.
$$
For $0<\tau\le1/7$, we have $
\tau^3\le\frac{\tau}{49}$, $\tau^2\le\frac{\tau}{7}$,
and hence
$$
H'(\tau)
\ge
\frac{22510}{7}\tau+\frac{225}{2}>0.
$$
Thus $H$ is increasing on $(0,1/7]$. Since
$H\left(\frac17\right)
=-\frac{9559}{33614}<0,$
we obtain $P_{1,c}(\tau)\le H(\tau)<0.$ Therefore,
$$
P_{k,c}(\tau)<0
$$
for every admissible $k,c,\tau$. Since
$T_{k,c}(\tau)>0$, \eqref{eq:G0_factor} gives
$$
G_0>0.
$$
Together with \eqref{eq:Gbeta_G0}, yields
$$
G_\beta>G_0>0.
$$

Finally, from \eqref{eq:T1T3_reduction} we obtain
$$
\begin{aligned}
F_\beta(1-x_3)\left[F_\beta(x_2)
+
\frac{\Psi_\beta(1-x_2)}{F_\beta(x_1)}\right]
\ge
1+F_\beta(1-x_3)G_\beta
>1.
\end{aligned}
$$
This proves the lemma.
\end{pf}

\begin{lem}\label{Last_lemma_for_p/q}
    Let $\frac35 < \frac pq < \frac23$ and $0 \le \xi \le \frac{1}{2p}$. Define $$A = q-p-1, ~ B = 2q-3p-1, ~ C = 3q-4p-1.$$
    If $(\xi, 2q-3p) \neq (0, 1)$, then
    $$F_\beta(1-\xi_B) + \frac{\Psi_\beta(\xi_B)}{F_\beta(1-\xi_C)} > \frac{1}{F_\beta(\xi_A)}.$$
\end{lem}

\begin{pf}
    For simplicity, let $a = \xi_A$, $b = \xi_B$, and $c = \xi_C$. Then $A+C = p+2B$ and 
    $$a+c = 2\xi + \frac{A+C}{p} = 1 + 2\left(\xi + \frac{B}{p}\right) = 1+2b.$$
    Define
    $$E_\beta = F_\beta(1-b) + \frac{\Psi_\beta(b)}{F_\beta(a-2b)} - \frac{1}{F_\beta(a)}.$$
    Since  $\Psi_\beta(b) \ge 1$ for $b \ge 0$ from Lemma \ref{Psi_function}, we have 
    $$E_\beta \ge F_\beta(1-b) + \frac{1}{F_\beta(a-2b)} - \frac{1}{F_\beta(a)}.$$
    Since $a-2b=1-c>0$ implies $0<a-2b < a<1$, from Lemma \ref{lem:reciprocal_gap_monotone} we have
    $$\frac{1}{F_\beta(a-2b)} - \frac{1}{F_\beta(a)} > \frac{1}{F_0(a-2b)} - \frac{1}{F_0(a)}.$$
    Since $F_\beta(1-b) > F_0(1-b)$ from Lemma \ref{Fbeta_beta_monotonicity},  to prove $E_\beta > 0$, it is sufficient to show
    $$F_0(1-b) + \frac{1}{F_0(a-2b)} - \frac{1}{F_0(a)} > 0.$$
    Let $U_0(x) = \frac{1}{F_0(x)} = \frac{x^2}{(1-x)(3x+1)}$. Then
    $$U_0'(x) = \frac{2x(x+1)}{(1-x)^2(3x+1)^2},$$
    which is strictly increasing on $(0,1)$. Moreover, for $\xi \le \frac{1}{2p}$ we have 
    $2a-b = 1 - \frac{1}{p} + \xi < 1$, which implies $a<\frac{1+b}{2}$ . Applying the Fundamental Theorem of Calculus 
    \begin{align}\label{bound_last_lem1}
        U_0(a) - U_0(a-2b) = \int_{a-2b}^{a} U_0'(x) \, dx < 2b \, U_0'\left(\frac{1+b}{2}\right).
    \end{align}
    Since $p/q > 3/5$, we have 
    $$0<b = \xi + \frac{2q-3p-1}{p} \le \frac{1}{2p} + 2\frac{q}{p} - 3 - \frac{1}{p} < \frac{1}{3}.$$
    Further, $F_0(1-b) = \frac{b(4-3b)}{(1-b)^2}$ and 
    $$F_0(1-b) - 2b \, U_0'\left(\frac{1+b}{2}\right) = -\frac{b\left(27b^3 + 70b^2 + 19b - 52\right)}{(1-b)^2(3b+5)^2}.$$
    The polynomial $P(b) = 27b^3 + 70b^2 + 19b - 52$ is strictly increasing for $b > 0$ and $P\left(\frac{1}{3}\right) = -\frac{332}{9} < 0.$
    Since $P(1/3)$ is strictly negative, $P(b) < 0$ for all $0 < b < 1/3$. Consequently, 
    \begin{align}\label{bound_last_lem2}
        F_0(1-b) > 2b \, U_0'\left(\frac{1+b}{2}\right).
    \end{align}
    Combining this with \eqref{bound_last_lem1} yields
    $$F_0(1-b) > U_0(a) - U_0(a-2b).$$
    Substituting the definition of $U_0(x)$, this is exactly equivalent to
    $$F_0(1-b) + \frac{1}{F_0(a-2b)} - \frac{1}{F_0(a)} > 0.$$
    Therefore, $E_\beta > 0$, and this completes the proof.
\end{pf}

 \section*{Acknowledgments}
The author thanks Professor Karlheinz Gröchenig for his valuable comments and suggestions on an earlier version of this manuscript.

\section*{Data Availability}
The author generated the figures using MATLAB. The data that support the findings of this study are available within the article and its supplementary material.
\section*{Declarations}
\textbf{Conflict of Interest}: The author has no conflicts to disclose.

\bibliographystyle{plain}

\bibliography{totallypositive}
\end{document}